\documentclass{article}
\usepackage[utf8]{inputenc}

\usepackage[margin=1in]{geometry}

\usepackage{microtype}
\usepackage{graphicx}
\usepackage{subfigure}
\usepackage{booktabs} %

\usepackage{amsmath}

\usepackage[colorlinks=true,allcolors=blue,hyperfootnotes=false]{hyperref}
\usepackage[capitalise]{cleveref}
\usepackage[numbers]{natbib}
   \makeatletter
\def\@fnsymbol#1{\ensuremath{\ifcase#1\or \dagger \or \ddagger \or \ddagger \or
   \mathsection\or \mathparagraph\or \|\or **\or \dagger\dagger
   \or \ddagger\ddagger \else\@ctrerr\fi}}
    \makeatother

\title{The Pendulum Arrangement: \\ Maximizing the Escape Time of Heterogeneous Random Walks}
\date{}

\author{%
Asaf Cassel%
\thanks{School of Computer Science, Tel Aviv University, Israel. 
Correspondence: \texttt{acassel@mail.tau.ac.il}.
}
\and
Shie Mannor%
\thanks{Faculty of Electrical Engineering, Technion Institute of
Technology, Israel. Correspondence: \texttt{shiemannor@gmail.com}, \texttt{guytenn@gmail.com}.
}
\\~\\
\textit{Names ordered alphabetically}
\and
Guy Tennenholtz$^\ddagger$%
}
\RequirePackage{mathtools}
\RequirePackage{dsfont}
\RequirePackage{delimset}

\newcommand{\ceil}[2][*]{\delim\lceil\rceil#1{#2}}
\newcommand{\floor}[2][*]{\delim\lfloor\rfloor#1{#2}}
\newcommand{\indEvent}[2][*]{\mathds{1}{\brk[c]#1{#2}}}

\newcommand{\RR}[1][]{\mathbb{R}^{#1}}

\DeclareMathOperator*{\argmax}{arg\,max}

\DeclarePairedDelimiterX\setDef[1]\lbrace\rbrace{\def\given{\;\delimsize\vert\;}#1}

\usepackage{amsthm}
\usepackage{amssymb}
\usepackage{thmtools}

\declaretheoremstyle[
	    spaceabove=\topsep, 
	    spacebelow=\topsep, 
	    bodyfont=\normalfont\itshape,
    ]{theorem}

\declaretheorem[style=theorem,name=Theorem]{theorem}

\declaretheoremstyle[
	    spaceabove=\topsep, 
	    spacebelow=\topsep, 
	    bodyfont=\normalfont,
    ]{definition}

\declaretheoremstyle[
        spaceabove=\topsep, 
        spacebelow=\topsep, 
        bodyfont=\normalfont,
        notefont=\normalfont\bfseries,
        notebraces={}{},
        qed=$\blacksquare$, 
    ]{proofstyle}
\declaretheorem[style=proofstyle,numbered=no,name=Proof]{proof}

\declaretheorem[style=theorem,sibling=theorem,name=Lemma]{lemma}

\declaretheorem[style=theorem,sibling=theorem,name=Proposition]{proposition}

\declaretheorem[style=theorem,numbered=no,name=Theorem]{theorem*}
\declaretheorem[style=theorem,numbered=no,name=Lemma]{lemma*}
\declaretheorem[style=theorem,numbered=no,name=Corollary]{corollary*}
\declaretheorem[style=theorem,numbered=no,name=Proposition]{proposition*}
\declaretheorem[style=theorem,numbered=no,name=Claim]{claim*}
\declaretheorem[style=theorem,numbered=no,name=Fact]{fact*}
\declaretheorem[style=theorem,numbered=no,name=Observation]{observation*}
\declaretheorem[style=theorem,numbered=no,name=Conjecture]{conjecture*}

\declaretheorem[style=definition,sibling=theorem,name=Definition]{definition}
\declaretheorem[style=definition,sibling=theorem,name=Remark]{remark}

\declaretheorem[style=definition,numbered=no,name=Definition]{definition*}
\declaretheorem[style=definition,numbered=no,name=Remark]{remark*}
\declaretheorem[style=definition,numbered=no,name=Example]{example*}
\declaretheorem[style=definition,numbered=no,name=Question]{question*}
\declaretheorem[style=definition,numbered=no,name=Assumption]{assumption*}

\newcommand{\escapeTimeOf}[2]{\tau\brk[c]*{#1; #2}}
\newcommand{\expectedEscape}[2]{\mathbb{E}\tau\brk[c]*{#1; #2}}

\newcommand{\varP}{p}
\newcommand{\varPof}[1]{\varP_{#1}}
\newcommand{\varA}{x}
\newcommand{\varAof}[1]{\varA_{#1}}
\newcommand{\varAlen}{d}

\newcommand{\pendOf}[1]{{#1}_{pend}}
\newcommand{\sortOf}[1]{{#1}_{sort}}

\newcommand{\permSig}{\sigma}
\newcommand{\permSigof}[1]{\permSig\brk*{#1}}

\newcommand{\permInv}{\sigma_{\text{mirror}}}
\newcommand{\permInvOf}[1]{\permInv\brk*{#1}}
\newcommand{\permImp}[1][l]{\permSig_{#1}}
\newcommand{\permImpOf}[2][l]{\permImp[#1]\brk*{#2}}

\newcommand{\permSortToPend}{\theta}
\newcommand{\permSortToPendOf}[1]{\theta\brk*{#1}}
\newcommand{\permSortToPendInv}{\theta^{-1}}
\newcommand{\permSortToPendInvOf}[1]{\theta^{-1}\brk*{#1}}

\newcommand{\permTilde}[1][l]{\tilde{\permSig}_{#1}}
\newcommand{\permTildeOf}[2][l]{\permTilde[#1]\brk*{#2}}

\newcommand{\allPermSet}{\Sigma}

\newcommand{\val}{J}
\newcommand{\valof}[1][\varA]{\val\brk*{#1; m}}

\newcommand{\win}[2][k]{W^{(#1)}_{#2}}
\newcommand{\winOf}[3][k]{\win[#1]{#2}\brk*{#3}}

\begin{document}

\maketitle

\begin{abstract}
    We identify a fundamental phenomenon of heterogeneous one dimensional random walks: the escape (traversal) time is maximized when the heterogeneity in transition probabilities forms a pyramid-like potential barrier. This barrier corresponds to a distinct arrangement of transition probabilities, sometimes referred to as the {\em pendulum arrangement}. We reduce this problem to a sum over products, combinatorial optimization problem, proving that this unique structure always maximizes the escape time. This general property may influence studies in epidemiology, biology, and computer science to better understand escape time behavior and construct intruder-resilient networks.
\end{abstract}

\section{Introduction}
\label{sec:intro}

Estimating the escape behavior of random walks has been an important performance indicator in fields such as biology \cite{smoluchowski1916drei,pulkkinen2013distance}, epidemiology \cite{lloyd2001viruses,hufnagel2004forecast}, cosmology \cite{krioukov2012network}, computer science \cite{lorek2017generalized}, and more \cite{rice1985diffusion,tuckwell1988introduction,carreras2002critical,benichou2005optimal}. Maximizing the escape time plays a crucial role in containing the spread of diseases or computer viruses \cite{lloyd2001viruses,pastor2001epidemic}, where the probability of an epidemic outbreak is closely related to properties of the contact network \cite{tennenholtz2020sequential}. In this work we identify a phenomenon related to the \textit{exact} escape time of a heterogeneous random walk on the finite line. Specifically, we show that the escape time is always maximized by a unique structure of transition probabilities, also known as the ``Pendulum Arrangement".

The characteristics of escape times of random walks have been extensively studied under the names of first passage time, escape times, and hitting times. While analytical formulations of the escape time have been established \cite{noh2004random,barrera2009abrupt,ding2018first}, their analysis has been mostly based on mean-field theory, asymptotic characteristics, and approximations \cite{noh2004random,condamin2007first,barrera2009abrupt,fronczak2009biased,tejedor2012random,hwang2012first,lee2014estimating,godec2016first}. Also related to our work, are studies on the speed of random walks in random environments \cite{solomon1975random,takacs2001more,mayer2004limit,procaccia2012need}. Specifically, \cite{procaccia2012need} show that the speed is minimized asymptotically by equally spaced drifts on the line. In contrast, our work takes an exact, combinatorial view of the problem, revealing an intrinsic feature of the maximum escape time in the general setting of an arbitrary heterogeneous random walk.

We consider a heterogeneous random walk on a finite line \cite{barrera2009abrupt}. Given a vector $\varP = (\varPof{1}, \hdots, \varPof{\varAlen})$ of~$\varAlen$ transition probabilities, the process, as depicted in \cref{fig: random walk}, starts at position $0$, moves backward with probability $\varPof{i}$ (reflecting at $0$), forward with probability $1-\varPof{i}$, and ends once it reaches position $\varAlen+1$. 
Our goal is to rearrange the elements of the vector $\varP$ (corresponding to rearranging the transition probabilities of moving backward on the line), so as to maximize the expected escape time of the random walk, namely, the time to reach position $\varAlen+1$ for the first time. 
Conceptually, we wish to form a potential barrier under a fixed budget, but are unsure where to place the barrier on the line.

\begin{figure} %
\includegraphics[width=\textwidth]{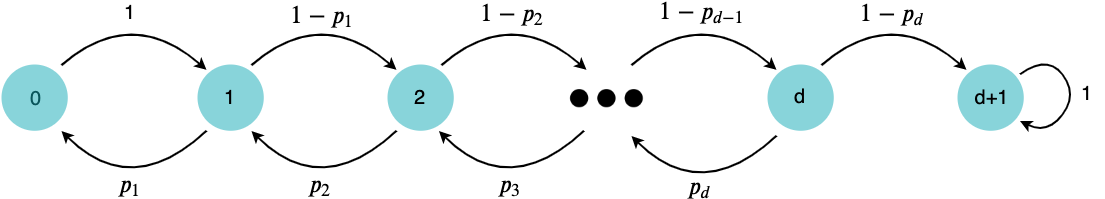}
\centering
\caption{\label{fig: random walk} The heterogeneous random walk process. 
}
\end{figure}

It is not clear a-priori whether the structure of this barrier has a closed form solution as it may depend on delicate relationships between the given probabilities. 
Intuitively, one might choose to arrange the transition probabilities in decreasing or increasing order. Here, an increasing order of the probabilities corresponds to forming a potential barrier toward the end of the line, reinforcing nodes in the vicinity of the termination node, whereas, a decreasing order corresponds to forming a barrier at the beginning of the line. Perhaps surprisingly, neither arrangement would  maximize the escape time. 

To obtain some intuition, consider an ascending order, where the highest probability is placed last. Notice that position $\varAlen$ is reached only after visiting position $\varAlen-1$, i.e., the second to last position is always visited more than the last position. It is thus unreasonable to place the highest probability last, as this would only decrease the expected escape time because it will be used less often. A similar argument can be made for a descending order, by switching between the first two probabilities. We will make this intuition precise in our complete derivation. 

Our main result shows that there is a {\em unique optimal} order of the transition probabilities that does not depend on their absolute, but rather their relative value, i.e., their sorted order. This also implies that changing the probabilities in a way that does not change their sorted order does not change the optimal arrangement. More specifically, we prove that the optimal order of the probabilities is such that they form a special pyramid-like shape, sometimes referred to as the pendulum arrangement (see \cref{fig: pendulum arrangement}), where the highest probability is placed in the middle. 

Finally, we formulate a continuous optimization variant of the problem, where the transition probabilities are optimized under limited budget constraints. We show that our main result can greatly diminish the complexity of finding an optimal solution. We also provide numerical experiments that illustrate the potential gains of using the pendulum arrangement, and discuss possible alternative statistics, including the minimum escape time.

\section{Problem Statement}
A vector $\varP = \brk*{\varPof{1}, \hdots, \varPof{\varAlen}} \in \brk*{0,1}^\varAlen$ of transition probabilities defines a heterogeneous random walk on a finite line of $\varAlen + 2$ states, as depicted in \cref{fig: random walk}. Formally, this process is defined by the following Markov chain. Let $\mathcal{M}^{\varP} = \brk[c]*{X_t^{\varP}}_{t=1}^\infty$ where $ X_t^{\varP} \in \setDef{0, 1, \hdots, \varAlen, \varAlen+1}$ is a random process that satisfies
\begin{equation*}
    P(X_{t+1}^{\varP} = j | X_t^{\varP} = i)
    =
    \begin{cases}
        \varPof{i} &,1 \leq i \leq \varAlen, j=i-1 \\
        1-\varPof{i} &,1 \leq i \leq \varAlen, j=i+1 \\
        1 &,(j = 1 \land i = 0) \lor (j = i = \varAlen+1) \\
        0 &,\text{otherwise}.
    \end{cases}
\end{equation*}
We define the escape time $\escapeTimeOf{\varP}{k}$ as the arrival time of $X_t^{\varP}$ to the termination state $\varAlen + 1$ given that it started at state $k$, i.e.,
\begin{equation*}
    \escapeTimeOf{\varP}{k}
    \coloneqq
    \min \brk[c]*{ t : X_t^{\varP} = \varAlen + 1; X_0^{\varP} = k}.
\end{equation*}
Our goal is to find the arrangement of the elements of $\varP$ that maximizes the expected escape time starting at state~$X_0 = 0$. Formally, let $\allPermSet$ be the set of permutations on $\brk[c]*{1, \ldots, \varAlen}$, i.e., $\permSig \in \allPermSet$ is a bijective mapping of $\brk[c]*{1, \ldots, \varAlen}$ onto itself. A vector $q = \permSig \varP$ is a permutation of the elements of $\varP$ defined as $q_i = \varPof{\permSigof{i}}$. Our goal is to find a permutation $\permSig^* \in \allPermSet$ such that
\begin{equation}
    \label{eq: max escape time}
    \tag{P1}
    \permSig^* \in \arg\max_{\permSig \in \allPermSet} \expectedEscape{\permSig \varP}{0}.
\end{equation}
In what follows we will show that $\permSig^*$ admits a unique solution that maps large transition values to the center, and small values to the edges of the line.

\begin{figure} %
\includegraphics[width=\textwidth]{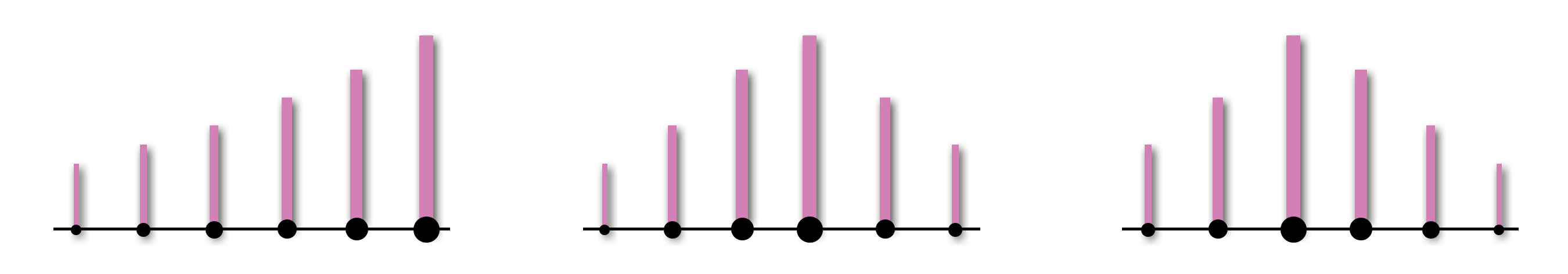}
\centering
\caption{\label{fig: pendulum arrangement} (left) Sorted arrangement. (middle) Pendulum arrangement. (right) Pendulum mirror image.}
\end{figure}

\section{Main Result}

This section states our main result, showing the optimal solution to Problem~\eqref{eq: max escape time} satisfies a unique symmetric arrangement, known as the pendulum arrangement, or its mirror. To that end, we define the mirror permutation $\permInv$, which reverses the vector it operates on.
\begin{definition}[Mirror Permutation] \label{def: mirror perm}
The mirror permutation is defined by $\permInvOf{i} = \varAlen + 1 - i$.
\end{definition}
\noindent Next, we define the pendulum arrangement.
\begin{definition}[Pendulum Arrangement]
\label{def:pendulum arrangement}
We say $\varA \in \RR[\varAlen]$ satisfies the pendulum arrangement if
\begin{align*}
    \varAof{i} \leq \varAof{d + 1 - i}
	,\;
	\forall 1 \leq i \leq \floor{\frac{\varAlen}{2}}
	\qquad\quad,\;
	\varAof{d + 1 - i} \leq \varAof{i + 1}
	,\;
	\forall 1 \leq i \leq \floor{\frac{\varAlen - 1}{2}}
	.
\end{align*}
We say $\pendOf{\varA}$ is a pendulum arrangement of $\varA$ if $\exists \permSig \in \allPermSet$ such that $\pendOf{\varA} = \permSig \varA$ and $\pendOf{\varA}$ is a pendulum arrangement.
\end{definition}
The pendulum arrangement has a special pyramid-like shape, as depicted in \cref{fig: pendulum arrangement}. Notice that traversing over its elements in descending order creates a pendulum-like motion hence explaining the name. Intuitively, the pendulum arrangement of a vector $\varA \in \RR[d]$ can be constructed by first sorting $\varA$ in decreasing order, and then placing the elements of the sorted array such that the largest element is in the middle, the next element to its left, the following element to its right, repeating this process until all elements have been placed in a pendulum-like ordering. This observation is made formal by the following lemma, which relates the pendulum arrangement to the sorted arrangement.
\begin{lemma}
\label{lemma:sort to pend}
    For $\varA \in \RR[\varAlen]$ let $\sortOf{\varA}$ be the result of sorting the elements of $\varA$ in ascending order.
    Define
    \begin{equation*}
        \permSortToPendOf{j} 
        =
        \begin{cases}
            2j-1 &, j \leq \frac{d+1}{2} \\
            2(d+1-j) &, \text{otherwise},
        \end{cases}
    \end{equation*}
    then $\pendOf{\varA}$ is uniquely defined and satisfies $\permSortToPend \sortOf{\varA} = \pendOf{\varA}$.
\end{lemma}

\noindent The proof of the lemma is technical and deferred to \cref{sec:proofOfSortToPend}. We are now ready to state our main result.

\begin{theorem}[Main Result]
\label{thm:main}
$\permSig^* \varP$ maximizes the expected escape time, i.e., solves Problem~\eqref{eq: max escape time}, if and only if it is ordered according to the pendulum arrangement, ${\permSig^*\varP = \pendOf{\varP}}$, or its mirror ${\permSig^*\varP = \permInv\pendOf{\varP}}$.
\end{theorem}
In other words, solving Problem~\eqref{eq: max escape time} reduces to finding a pendulum arrangement of the elements of~$\varP$, which is immediately obtained from their sorted order. Moreover, this solution is unique up to its mirror.

\newpage
\section{Proof of Main Result}

\begin{figure} %
\centering
\includegraphics[width=\textwidth]{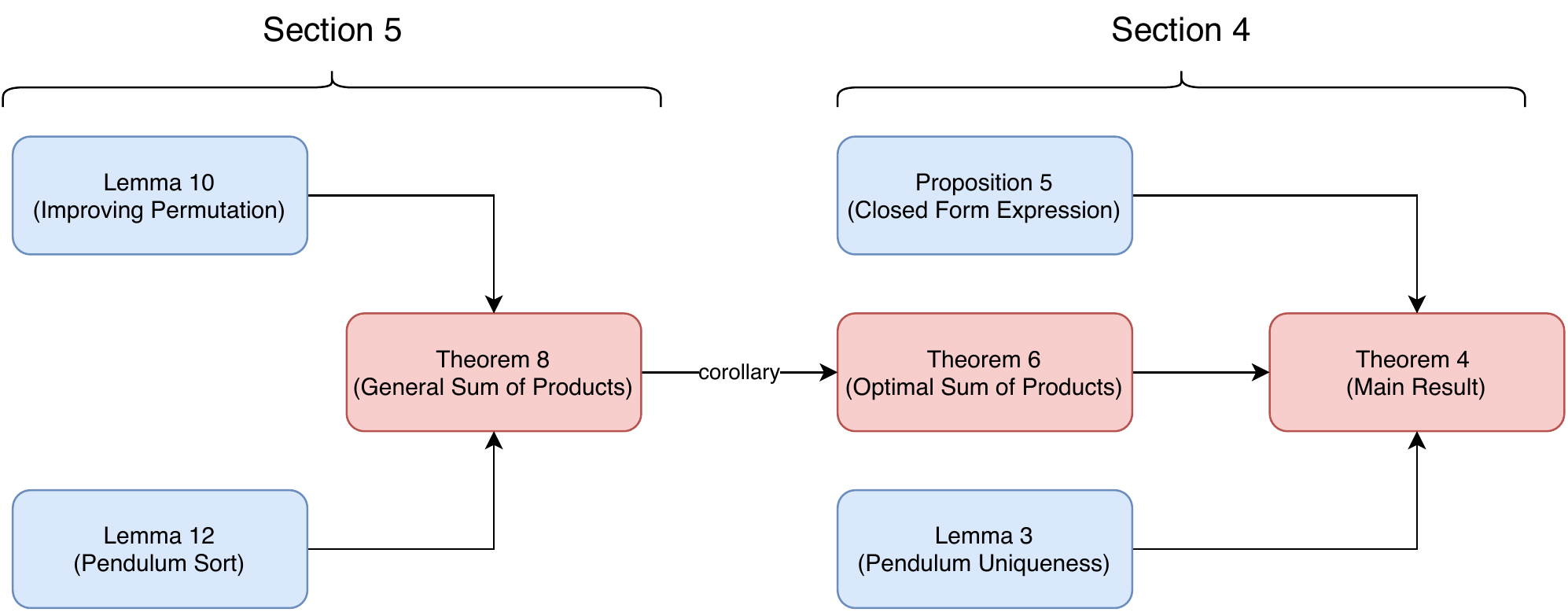}
\caption{
\label{fig: theorems flow chart}
A flowchart for proving \cref{thm:main} (Main Result). In red are the theorems used to prove the final result, and in blue the main supporting lemmas.
}
\end{figure}

The proof of \cref{thm:main} consists of two parts, as seen in \cref{fig: theorems flow chart}. In this section we focus on the right part of \cref{fig: theorems flow chart}, showing a closed form expression for the expected escape time $\expectedEscape{\varP}{0}$, which reduces the problem to maximizing a sum over products. We then prove that the pendulum arrangement maximizes this sum of products, thus concluding the proof. This second part, which is used here as a tool, is the heart of the problem and we discuss and explain its main ideas in \cref{sec:sum over products}.

The following proposition derives a closed form expression for $\expectedEscape{\varP}{0}$. Its proof uses a direct inductive claim and is provided in \cref{sec:proofOfEscapeFormula}.
\begin{proposition}[Closed Form Expression] \label{proposition:escapeTimeFormula}
We have that
\begin{align}
\label{eq: escape time formula}
    \expectedEscape{\varP}{0}
    =
    (\varAlen + 1)
    +
    2
    \sum_{m=1}^{\varAlen}
    \sum_{i=1}^{\varAlen - m + 1} 
    \prod_{j=i}^{i + m - 1} \frac{\varPof{j}}{1-\varPof{j}}.
\end{align}
\end{proposition}
One can immediately notice a symmetric property of $\expectedEscape{\varP}{0}$ in \cref{eq: escape time formula}. Specifically, it is invariant to the mirror permutation, i.e.,
\begin{equation*}
    \expectedEscape{\varP}{0} 
    = 
    \expectedEscape{\permInv \varP}{0},
    \;\;
    \forall \varP \in \brk*{0,1}^{\varAlen}.
\end{equation*}
This in turn implies that the pendulum arrangement and its mirror both achieve the same value, and thus proving that one of them is optimal will suffice to conclude \cref{thm:main} (Main Result).
Also note that this implies that ascending and descending orderings of the elements of $\varP$ achieve identical (yet sub-optimal) values. This fact is indicative of a symmetric characteristic of $\expectedEscape{\varP}{0}$ that foreshadows the underlying pendulum arrangement. 

Focusing on the sum over products term in \cref{eq: escape time formula}, we have the following theorem, which states that the pendulum arrangement is  its unique maximizer.

\begin{theorem}[Optimal Sum of Products]
\label{thm:optimal sum prod}
    For any $\varA \in \RR[d]_{++}$ we have that
    \begin{equation*}
            \permSig^* \in \argmax_{\permSig \in \allPermSet} \sum_{m=1}^{\varAlen}
        \sum_{i=1}^{\varAlen - m + 1} 
        \prod_{j=i}^{i + m - 1} \varAof{\permSigof{i}},
            \iff
            \permSig^* \varA \in \brk[c]*{\pendOf{\varA}, \permInv \pendOf{\varA}}
            .
    \end{equation*}
\end{theorem}
As we will show next, combining \cref{thm:optimal sum prod} with \cref{lemma:sort to pend} and \cref{proposition:escapeTimeFormula} yields a straightforward proof for \cref{thm:main} (Main Result). The proof of \cref{thm:optimal sum prod} is the crux of this work and is outlined in the following section. Before diving into its details, we show how it can be used to prove \cref{thm:main} (Main Result).

\begin{proof}[of \cref{thm:main} (Main Result)]
    Consider the expression for $\expectedEscape{\varP}{0}$ in \cref{eq: escape time formula}. Denoting ${\varA = \varP / \brk*{1 - \varP}}$, where the equality is element-wise, we have that
	\begin{equation*}
    	\permSig^*
    	\in
    	\argmax_{\permSig \in \allPermSet} \expectedEscape{\permSig \varP}{0}
    	\iff
    	\permSig^*
    	\in
    	\argmax_{\permSig \in \allPermSet} \sum_{m=1}^{\varAlen} \sum_{i=1}^{\varAlen - m + 1} \prod_{j=i}^{i+m-1} \varAof{\permSigof{i}}
    	.
	\end{equation*}
	Next, notice that $\varA \in \RR[\varAlen]_{++}$ and so using \cref{thm:optimal sum prod} we have that
	\begin{equation*}
        \permSig^* \in \argmax_{\permSig \in \allPermSet} \sum_{m=1}^{\varAlen} \sum_{i=1}^{\varAlen - m + 1} \prod_{j=i}^{i+m-1} \varAof{\permSigof{i}}
        \iff
        \permSig^* \varA \in \brk[c]*{\pendOf{\varA}, \permInv \pendOf{\varA}}
        .
    \end{equation*}
	Finally, notice that the function $f(y) = y / \brk*{1 - y}$ is strictly increasing in $[0, 1)$ and thus the sorted order of $\varP$ and $f(\varP) = \varA$ are the same. Since the pendulum arrangement only depends on this order (see \cref{lemma:sort to pend}), we get that
	$
	\permSig^* \varA \in \brk[c]*{\pendOf{\varA}, \permInv \pendOf{\varA}}
	\iff
	\permSig^* \varP \in \brk[c]*{\pendOf{\varP}, \permInv \pendOf{\varP}}
	,
	$ 
	and combining these arguments concludes the proof.
\end{proof}

\section{Sum Over Products}
\label{sec:sum over products}

In this section we will focus on proving \cref{thm:optimal sum prod}. We do so by considering each of the inner summations in \cref{thm:optimal sum prod}, reducing the problem further to individually maximizing each of the inner sums of products (\cref{thm:general opt sum prod}). We then move to define the improving permutation (\cref{def:improving perm}), a uniquely designed permutation that: (1) always improves the sum over products (\cref{lemma: impPerm}); and (2) converges after at most $\varAlen / 2$ applications to the pendulum (optimal) arrangement (\cref{lemma:pendulumSort}). These results will finalize the proof of \cref{thm:optimal sum prod}, thereby concluding the proof of \cref{thm:main} (Main Result).
To that end, we begin by focusing on the following construct.

\begin{definition}[Sum Over Products Value]
For any $\varA \in \RR[\varAlen]$ and $1 \le m \le \varAlen$ define the value function of $\varA$ for window size $m$ as
\begin{equation*}
	\valof[\varA] = \sum_{i=1}^{\varAlen - m + 1} \prod_{j=i}^{i + m - 1} \varAof{j}.
\end{equation*}
\end{definition}

The function $\valof[\varA]$ is a sum over all products of adjacent tuples of length $m$. For example, for $\varAlen = 5$ and $m = 3$ it can be explicitly written as
$
    \valof[\varA] 
    = 
    \varAof{1}\varAof{2}\varAof{3}
    +
    \varAof{2}\varAof{3}\varAof{4}
    +
    \varAof{3}\varAof{4}\varAof{5}.
$
Notice that the expression in \cref{thm:optimal sum prod} is in fact a summation of the sum over products value, $\valof[\varA]$, for various window sizes $1 \leq m \leq \varAlen$. \cref{thm:optimal sum prod} is thus an immediate corollary of the following, more general result.
\begin{theorem}[General Sum Over Products] \label{thm:general opt sum prod}
    For any $\varA \in \RR[\varAlen]_{++}$ we have that
    \begin{enumerate}
        \item (Sufficiency) $\forall\; 1 \le m \le \varAlen$,
            $
                \permSig^* \varA \in \brk[c]*{\pendOf{\varA}, \permInv\pendOf{\varA}}
                \implies
                \permSig^* \in \argmax_{\permSig \in \allPermSet} \valof[\permSig\varA];
            $
        \item (Necessity)
            $
                \permSig^* \in \cap_{m=1}^{\varAlen}\argmax_{\permSig \in \allPermSet} \valof[\permSig\varA]
                \implies
                \permSig^* \varA \in \brk[c]*{\pendOf{\varA}, \permInv\pendOf{\varA}}.
            $
    \end{enumerate}
\end{theorem}

\begin{figure}[t!]
\centering
\includegraphics[width=0.7\textwidth]{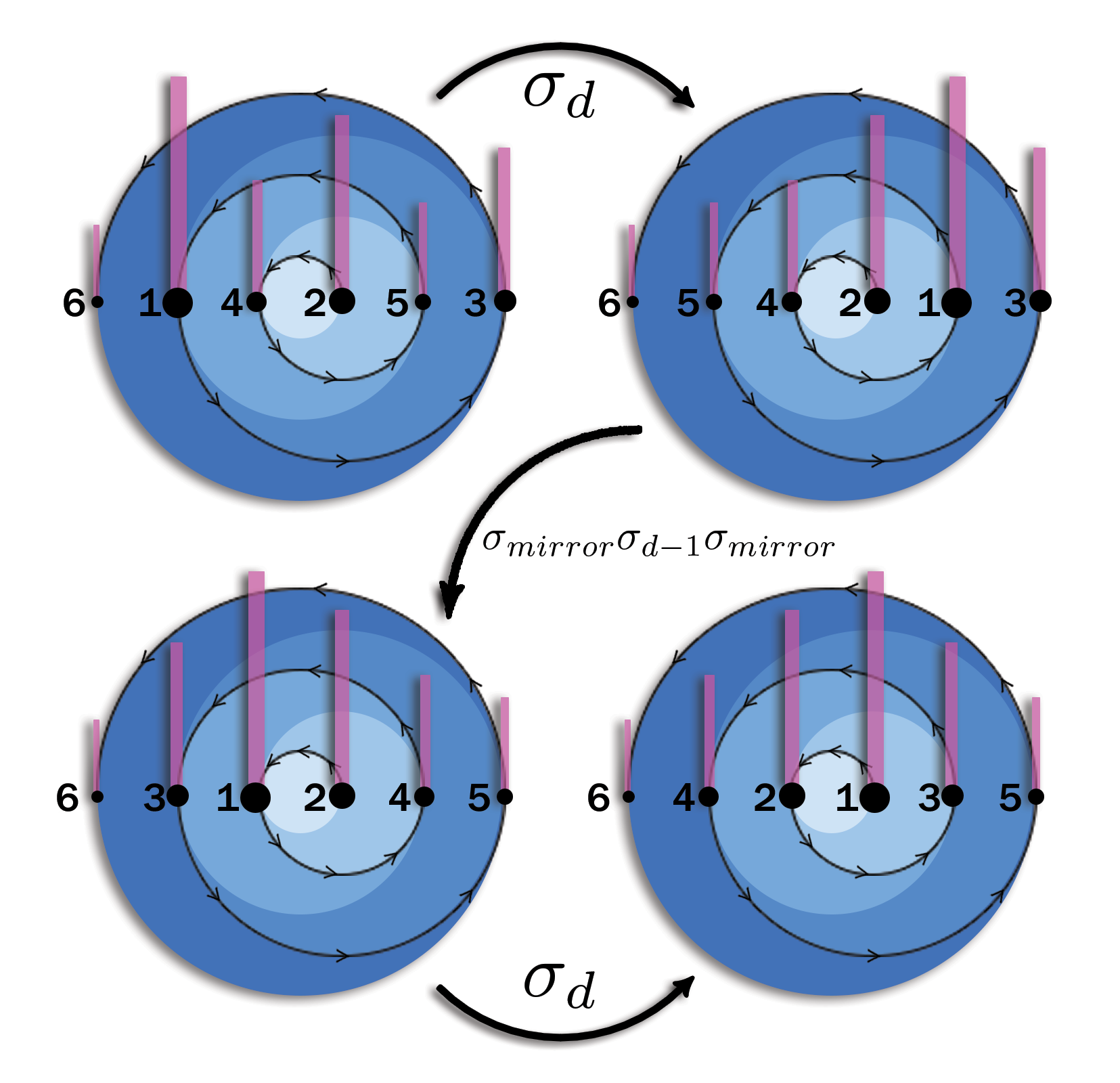}
\caption{\label{fig: improving permutation} An example of applying the improving permutations $\permImp[\varAlen]$ and $\permInv\permImp[\varAlen-1]\permInv$ iteratively on some given vector $\varA$. Plot shows direction in which switching of elements occur. Small values follow the circular arrows, whereas large values move in the reverse direction. Elements switch until reaching their final position in the pendulum arrangement.}
\end{figure}

\subsection{Improving Permutation}
The main tool for proving \cref{thm:general opt sum prod} is the following permutation.
\begin{definition}[Improving Permutation] \label{def:improving perm}
For ${l = 1,\ldots, \varAlen}$ define the $l^{th}$ improving permutation of a vector $\varA \in \RR[\varAlen]$ by
\begin{equation}
\label{eq:perm imp}
	\permImpOf{i}
	=  
	\begin{cases}
	{l + 1 - i}, &\quad \text{or}~ 
	\begin{aligned}
		&\varAof{i} > \varAof{l + 1 - i},~ i \le l / 2 \\
		&\varAof{i} < \varAof{l + 1 - i},~ l / 2 < i \le l
	\end{aligned} \\
	{i}, &\quad \text{otherwise}.
	\end{cases}
\end{equation}
\end{definition}
We note that the vector $\varA$, with respect to which $\permImp$ is defined, is always the vector it permutes. While it is not denoted explicitly in $\permImp$, its identity will always be clear from context.
The improving permutation, $\permImp$, compares elements across the symmetry axis $\brk*{l + 1} / 2$, and switches their positions such that the larger element is to the right of the symmetry axis (see \cref{fig: improving permutation}). Notice that this may result in up to $l / 2$ exchanges. While this may seem overly complicated, it is easy to give counter examples where any exchange of two elements will decrease the outcome (see \cref{remark: no simple permutation}).
As its name suggests, applying $\permImp$ to a vector increases its sum over products value, as shown by the following lemma. An exhaustive proof is provided \cref{sec:proofOfImpPermLemma}.

\begin{lemma}[Improving Permutation]
\label{lemma: impPerm}
	For all $\varA \in \RR[\varAlen]_{++}$ and $m, l \in \brk[c]*{1, \ldots, \varAlen}$ we have that
	\begin{equation*}
		\valof[\permImp\varA] \ge \valof[\varA]
		.
	\end{equation*}
	Moreover, if 
	$\permImp \varA \notin \brk[c]*{\varA, \permInv\varA}$ then there exists $m$ such that the inequality is strict.
\end{lemma}

\begin{figure} %
\centering
\includegraphics[width=0.4311\textwidth]{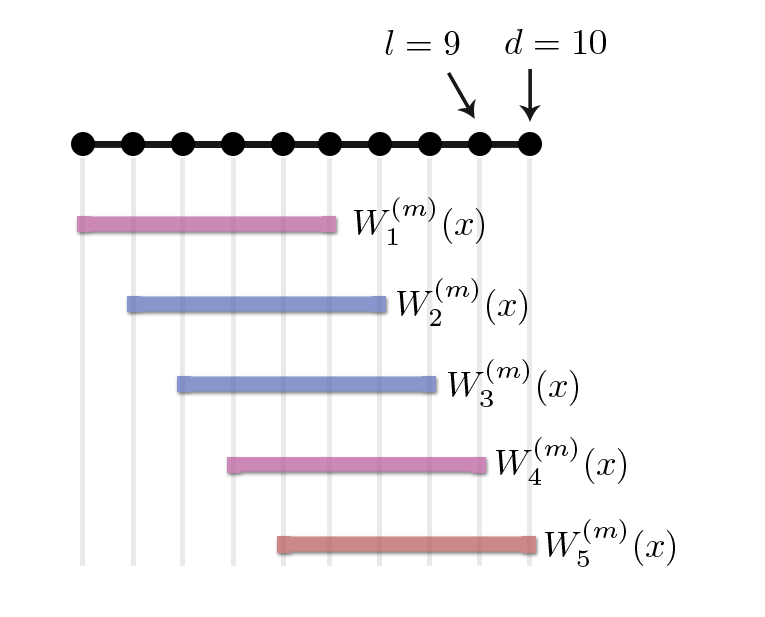}
\includegraphics[width=0.5589\textwidth]{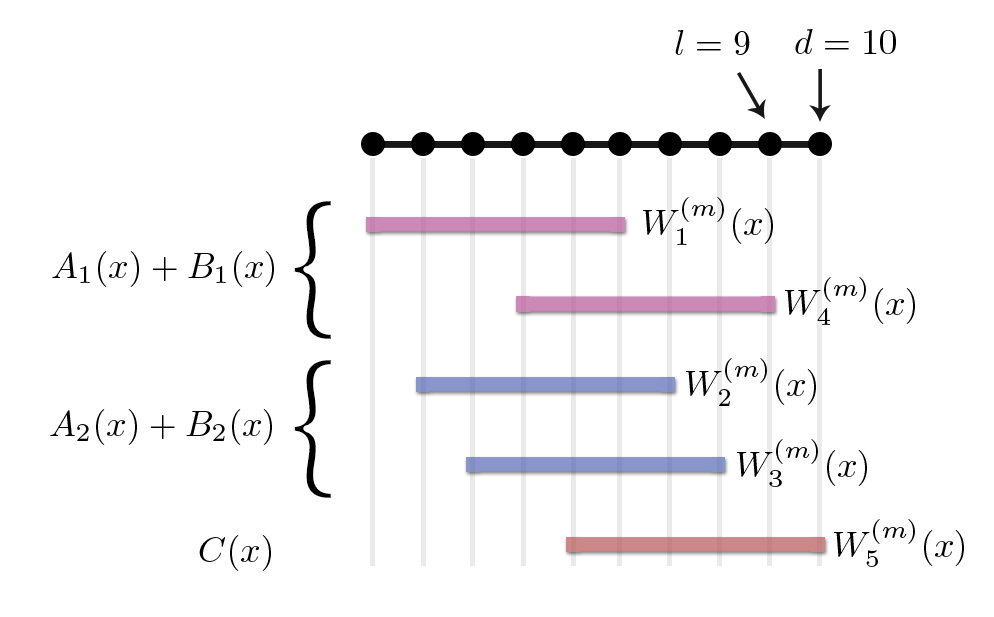}
\caption{
\label{fig: windows}
Depiction of windows as defined in the sketch proof of \cref{lemma: impPerm}. Both figures show windows for the case of $\varAlen = 10, m = 6$, with choice of $l = 9$. (left) Windows as they would be summed over in $\valof$. (right) Windows are reorganized to fit the structure of \cref{eq: reorg J}, which ensures the improving permutation increases each term individually.
}
\end{figure}

\begin{proof}[sketch of \cref{lemma: impPerm}]
    We begin by denoting the product over a ``window'' of size $m$ starting at $i$ by $\winOf[m]{i}{\varA} = \prod_{j=i}^{i + m - 1} \varAof{j}$, i.e., $\valof = \sum_{i=1}^{} \winOf[m]{i}{\varA}$. 
    With some algebra, we then show that
    \begin{equation}
    \label{eq: reorg J}
        \valof
        =
        \sum_{i=1}^{\ceil{\frac{l-m}{2}}} \brk[s]*{A_i\brk*{\varA} + B_i\brk*{\varA}}
        +
        C(x, m, l),
    \end{equation}
    where $A_i(\varA) = \winOf[m]{i}{\varA}$, $B_i(\varA) = \winOf[m]{(l + 2 - m) - i}{\varA}$, and
    $
        C(x, m, l)
        \approx
        \sum_{i = l + 2 - m}^{\varAlen - m + 1} \winOf[m]{i}{\varA}
    .$
    \cref{fig: windows} depicts an example of how \cref{eq: reorg J} reorganizes the elements of $\valof$. 
    
    In \cref{eq: reorg J}$, A_i$ and $B_i$ were chosen such that if $\varAof{j}$ participates in $A_i(\varA)$ then $\varAof{l+1 - j}$ participates in $B_i(\varA)$. Since these are the only kind of switches $\permImp[l]$ makes, we conclude that
    $
        A_i(\varA)B_i(\varA)
        =
        A_i(\permImp[l]\varA)B_i(\permImp[l]\varA)
        .
    $
    Since $\permImp[l]$ puts the larger element in $l + 1 - j$, i.e., in $B_i$, we also have that
    $
        B_i(\permImp[l]\varA) \ge \max\brk[c]*{A_i(\varA), B_i(\varA)},
    $
    with strict inequality if only some but not all of the elements were switched.
    Combining the last two claims, it is immediate to conclude that
    \begin{equation*}
        A_i\brk*{\permImp[l]\varA} + B_i\brk*{\permImp[l]\varA}
        \ge
        A_i\brk*{\varA} + B_i\brk*{\varA}.
    \end{equation*}
    This is equivalent to saying that elongating the longer side of a rectangle while maintaining its area fixed (by shortening the other side) increases its circumference. The above holds for the relevant indices and thus showing that
    $
        C(\permImp[l]\varA, m, l) \ge C(\varA, m, l)
    $
    concludes the proof. This is straightforward since $\permImp[l]$ essentially increases each of its terms individually. 
\end{proof}
\begin{remark}
\label{remark: no simple permutation}
When $m=2$, it is always possible to find a so called ``improving" permutation that only exchanges two elements; however, this is not the case for $m \ge 3$. To see this, take for example, the case of $\varP = (0.17, 0.64, 0.85, 0.71)$. Exhaustive search shows that this is the second to best ordering and thus any improvement must lead to one of the optimal orderings $(0.64, 0.85, 0.71, 0.17)$ or its mirror $(0.17, 0.71, 0.85, 0.64)$. Notice that any such permutation must indeed exchange more than two elements. In other words, there exists an initialization vector $\varP$ for which no ``simple" permutation (i.e., one which exchanges only two elements) could iteratively converge to the optimal ordering. This motivates the use of more elaborate improving permutations as proposed in \cref{def:improving perm}.
\end{remark}

\begin{figure} %
\centering
\includegraphics[width=\textwidth]{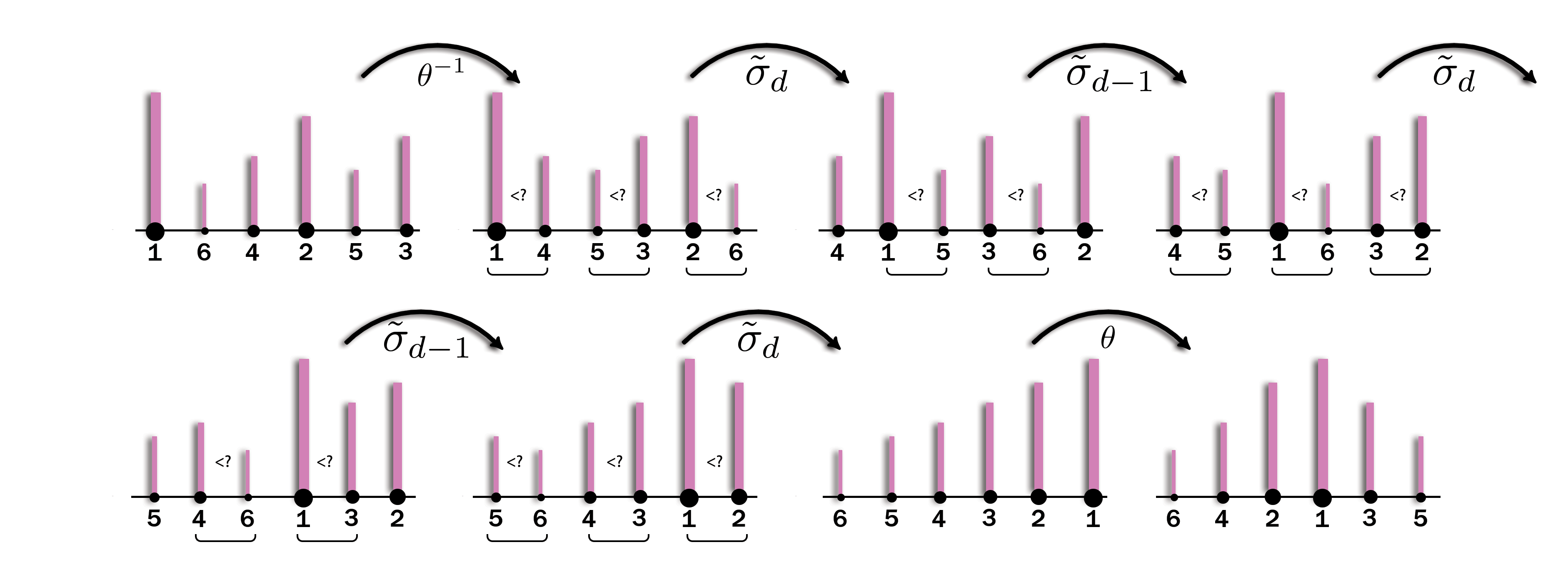}
\caption{\label{fig: pendulum sort} 
Pendulum Sort: An example of an application of the pendulum sort permutation $\permSortToPend \brk*{\permTilde[\varAlen - 1] \permTilde[\varAlen]}^{\frac{d}{2}} \permSortToPend^{-1}$ iteratively. $\permTilde[\varAlen]$ and $\permTilde[\varAlen - 1]$ compare element pairs, switching them whenever the left element is larger than the neighbor on its right. $\permTilde[\varAlen - 1]$ compares pairs of elements in even indices, whereas $\permTilde[\varAlen]$ compares them at odd indices. Note that application of $\permTilde[\varAlen]$ or $\permTilde[\varAlen - 1]$ on a sorted array is the identity permutation. }
\end{figure}

\subsection{Pendulum Sort}

Having established that $\permImp$ ($1 \le l \le \varAlen$) are always improving, we show that applying them consecutively converges to a pendulum arrangement. More specifically, the following lemma uses $\permImp[\varAlen], \permImp[\varAlen -1]$ and $\permInv$ (see \cref{def: mirror perm,def:improving perm}) to construct such a sequence. An exhaustive proof of the lemma is provided \cref{sec:proofOfPendSort}.

\begin{lemma}[Pendulum Sort]
\label{lemma:pendulumSort}
For all $\varA \in \RR[\varAlen], k \geq \frac{d}{2}$, we have that $(\permInv\permImp[\varAlen-1]\permInv\permImp[\varAlen])^k \varA = \pendOf{\varA}$.
\end{lemma}

\begin{proof}[sketch of \cref{lemma:pendulumSort}]
    Recall that $\permSortToPend$ from \cref{lemma:sort to pend} satisfies $\permSortToPend \sortOf{\varA} = \pendOf{\varA}$.
    We define $\permTilde[\varAlen], \permTilde[\varAlen-1]$ as follows
    \begin{align*}
        \permTilde[\varAlen] = \permSortToPend^{-1} \permImp[\varAlen] \permSortToPend,
        \qquad
        \permTilde[\varAlen - 1] = \permSortToPend^{-1} \brk*{\permInv \permImp[\varAlen - 1] \permInv} \permSortToPend,
    \end{align*}
    and a simple telescoping argument yields that
    \begin{equation*}
        \brk*{\permInv \permImp[\varAlen - 1] \permInv \permImp[\varAlen]}^k
        =
        \permSortToPend \brk*{\permTilde[\varAlen - 1] \permTilde[\varAlen]}^k \permSortToPend^{-1}.
    \end{equation*}
    We then show that for any $y \in \RR[\varAlen]$, $\brk*{\permTilde[\varAlen - 1] \permTilde[\varAlen] }^k y = \sortOf{y}$ for all $k \geq \frac{\varAlen}{2}$. Recalling \cref{lemma:sort to pend} and choosing $y = \permSortToPend^{-1} \varA$ concludes the proof. To show that $\brk*{\permTilde[\varAlen - 1] \permTilde[\varAlen] }^k y = \sortOf{y}$ we first find explicit expressions for $\permTilde[\varAlen], \permTilde[\varAlen - 1]$. These expressions are sorting procedures on the odd and even odd pairs of $\varA$ respectively. This means that applying them consecutively performs a sort of parallel bubble sort, which is depicted in \cref{fig: pendulum sort}. A simple analysis shows that this converges in $\varAlen^2$ steps and a more careful analysis gives the desired $\varAlen / 2$ steps.
\end{proof}

\begin{proof}[of \cref{thm:general opt sum prod}]
    First, recall that $\valof[\varA] = \valof[\permInv\varA]$ and so using \cref{lemma: impPerm}~(Improving Permutation) recursively we get that
    \begin{equation*}
        \valof[{(\permInv\permImp[\varAlen-1]\permInv\permImp[\varAlen])^{k} \varA}]
        \ge
        \valof,
        \;\forall k \geq 0.
    \end{equation*}
    Taking $k \ge \varAlen / 2$ and using \cref{lemma:pendulumSort}~(Pendulum Sort) we then get that
    $
        \valof[\pendOf{\varA}]
        \ge
        \valof[\varA],
    $
    and since this holds for any permutation of $\varA$, the first part of the proof is concluded.
    The uniqueness claim follows from the strict inequality condition of \cref{lemma: impPerm}~(Improving Permutation). More concretely, let 
    \begin{align*}
        \permSig^* 
        \in
        \cap_{m=1}^{\varAlen}\argmax_{\permSig \in \allPermSet} \valof[\permSig\varA]
        ,
    \end{align*}
    and assume in contradiction that $\permSig^* \varA \notin \brk[c]*{\pendOf{\varA}, \permInv \pendOf{\varA}}$. However, from \cref{lemma:pendulumSort}~(Pendulum Sort) we know that
    $
    {
        (\permInv\permImp[\varAlen-1]\permInv\permImp[\varAlen])^{\varAlen} \permSig^* \varA
        =
        \pendOf{\varA}
        ,
    }
    $
    and thus one of the terms composing
    $
        (\permInv\permImp[\varAlen-1]\permInv\permImp[\varAlen])^{\varAlen}
    $
    must change its input to something other than its mirror.
    The strict inequality condition of \cref{lemma: impPerm}~(Improving Permutation) then implies that there exists $m$ such that
    ${
        \valof[\pendOf{\varA}]
        >
        \valof[\permSig^* \varA]
        ,}
    $
     contradicting the optimality of $\permSig^*$.
\end{proof}

\begin{figure}[t!]
\includegraphics[width=0.7\textwidth]{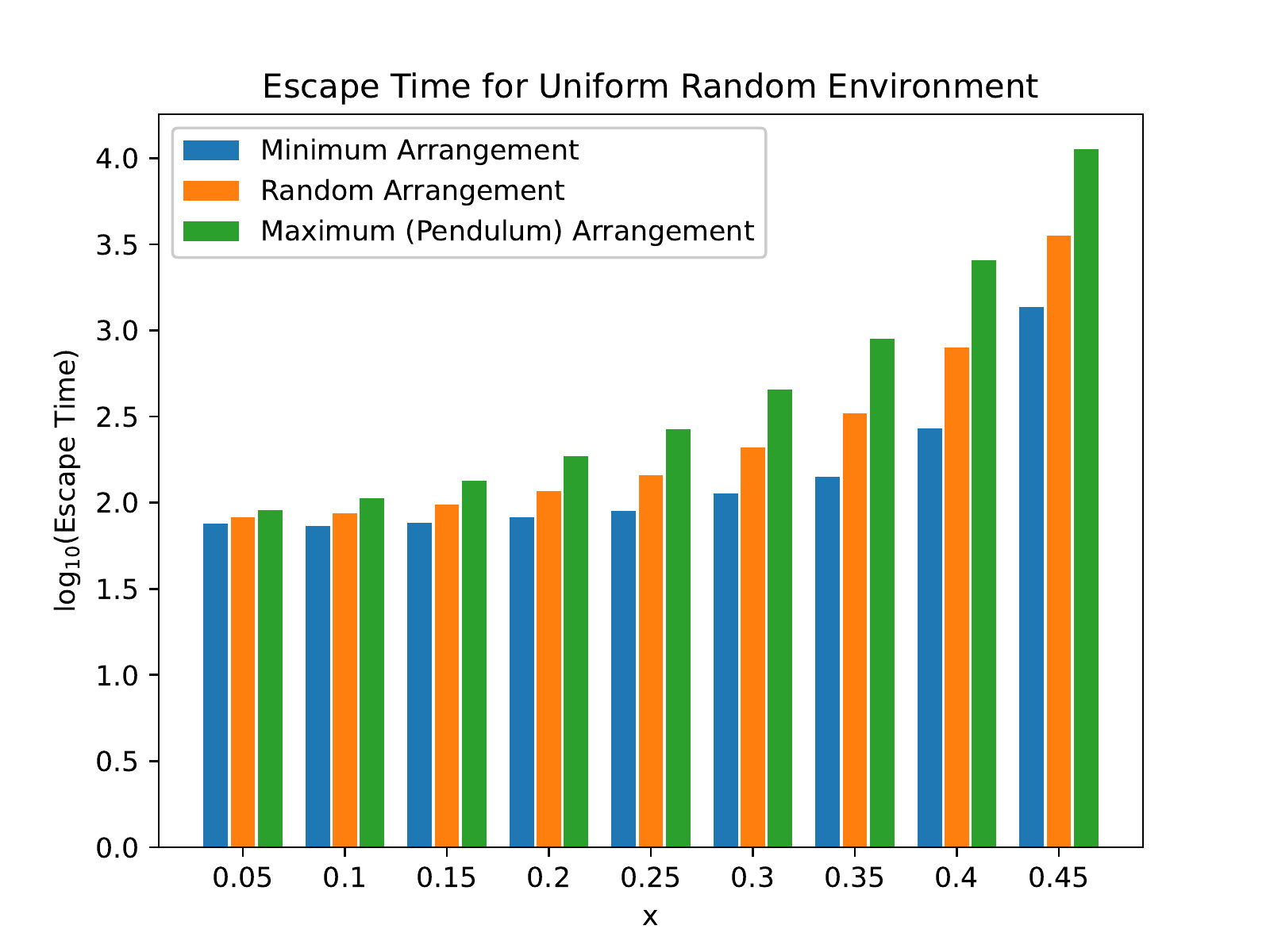}
\centering
\caption{\label{fig: min-max comparison}
Comparison of various arrangements and their escape times for a random walk on a line with $\varAlen = 8$ nodes in a random environment in which transition values were sampled i.i.d. from a uniform distribution in the interval $[0.5-x, 0.5+x]$ for various values of $x$. The random arrangement is the mean escape time taken w.r.t. the uniform measure over all possible permutations. The presented value for all statistics was averaged over 1000 different instantiations of the random environment. }
\end{figure}

\newpage
\section{Discussion and Future Work}
In this section we demonstrate a continuous extension to our main result, conduct numerical experiments on random environments that illustrate the significance of our findings, and discuss alternate statistics of the escape time.

\subsection{Continuous Weight Optimization}

We consider the following continuous optimization variant of the combinatorial problem~\eqref{eq: max escape time}:
\begin{equation}
    \label{eq: control problem}
    \tag{P2}
    \varP^*
    \in
    \arg\max_{p \in C} \expectedEscape{\varP}{0}
    ,
\end{equation}
where $C \subseteq \left[0,1\right)^{\varAlen}$ is a set of budget constraints on the transition probabilities. The difficulty of \eqref{eq: control problem} strongly depends on the structure of the set $C$. \cref{thm:main} (Main Result) implies that for
$
C_{\varP}
=
\brk[c]{\permSig p \;\big|\; \permSig \in \allPermSet}
,
$
\eqref{eq: max escape time} is efficiently solvable. The following proposition readily follows from \cref{thm:main} (Main Result), and extends it to a slightly more general class of constraints.
For $A \subseteq C$ let $\text{ext}(A)$ denote the extreme points of the convex hull of $A$, and
$\pendOf{A} = \brk[c]{\pendOf{\varP} \;\big|\; \varP \in A}$ (see \cref{def:pendulum arrangement}).
\begin{proposition}
\label{prop: extreme points}
For $C \subseteq \left[\frac{1}{2},1\right)^{\varAlen}$, if
$
     \pendOf{(\text{ext}(C))} \subseteq \text{ext}(C)
$
then $\exists \varP^* \in \pendOf{(\text{ext}(C))}$. 
\end{proposition}

In other words, if the pendulum arrangement is always an element of the extreme points of $C$ then the optimal solution to Problem~\eqref{eq: control problem} is an extreme point of $C$ which is ordered according to the pendulum arrangement. 
The proof of \cref{prop: extreme points} is provided in \cref{appendix: proof of prop extreme points} and uses the fact that $\expectedEscape{\varP}{0}$ is convex in $\varP$. This implies that there exists $\varP^* \in \text{ext}(C)$ and thus applying \cref{thm:main} (Main Result) with the assumed structure of $C$ concludes the proof. This result allows us to greatly reduce the search for an optimal solution. Particularly, it may reduce this search to a small constant number of possible candidates, as shown by the following example.

\paragraph{Example:} Assume a linear budget constraint of the form
\begin{equation*}
    C_{a, b}
    =
    \brk[c]*{\varP \in \brk[s]{0, a}^{\varAlen}
    ~\Big|~
    \norm{\varP}_1 \leq b},
\end{equation*}
where $a \in [0, 1)$.
Trivially, whenever $b \geq \varAlen a$ the optimal solution is given by the uniform vector ${\varP^* = [a, \hdots, a]}$. Yet, when $b < \varAlen a$, by \cref{prop: extreme points}, the optimal solution will be given by a pendulum arrangement over $\text{ext}(C_{a,b})$. This results in $b / a$ values of $a$ (up to a remainder term) placed in the center of the line. Concretely, $\varP^* = \pendOf{\varP}$ with
\begin{equation*}
    \varP = {\brk{\underbrace{a, \hdots, a}_{\floor{\frac{b}{a}}~\text{times}}, \text{mod}(b, a), 0, \hdots, 0}}.
\end{equation*}

\begin{figure}[t!]
\includegraphics[width=0.46\textwidth]{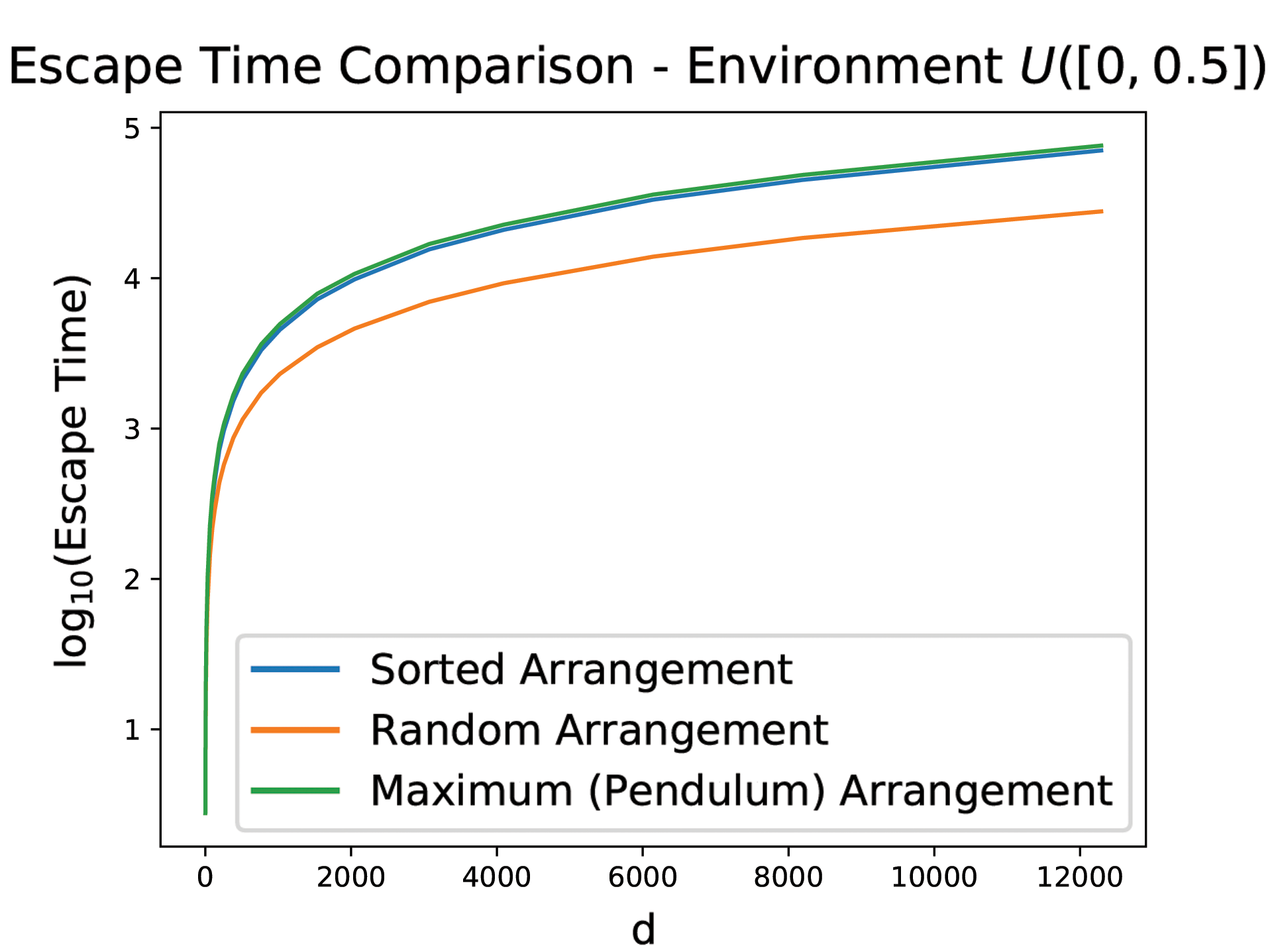}
\hspace{1cm}
\includegraphics[width=0.46\textwidth]{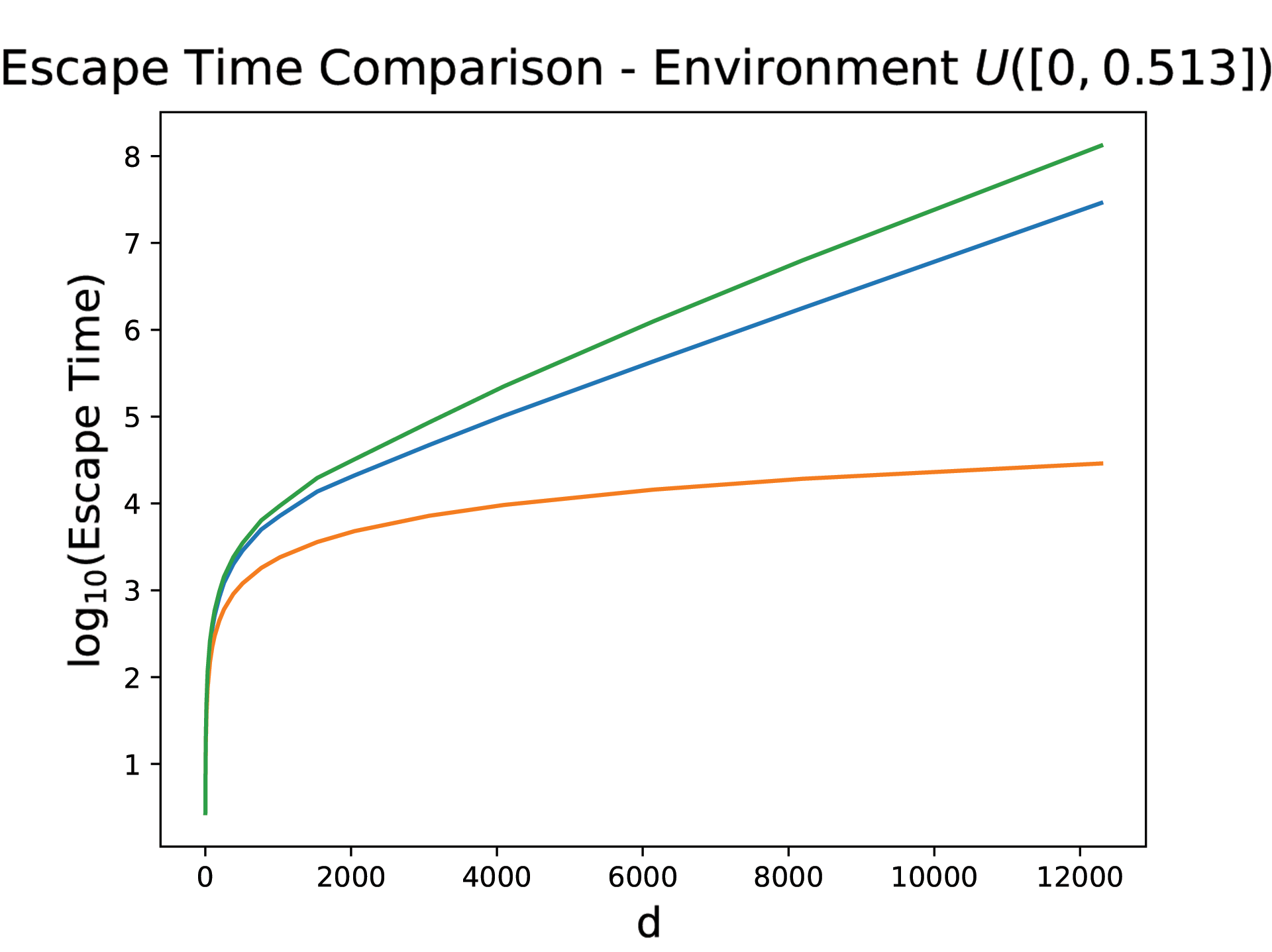}
\centering
\caption{\label{fig: d-comparison}
Escape time comparison of the maximal, sorted, and random arrangements as a function of the number of nodes $\varAlen$. (left) environment weights distributed $U([0, 0.513])$. (right) environment weights distributed $U([0, 0.5])$. Graphs display average result over 1000 environment instantiations where for each instantiation the random arrangement is calculated by averaging 1000 random permutations. }
\end{figure}

\subsection{Random Environments}
\cref{thm:main} (Main Result) shows that the pendulum arrangement yields the maximum expected escape time. In this section we perform several numerical experiments to give a more quantitative grasp of the behavior of the expected escape time under different arrangements: maximal (pendulum), minimal, sorted, and random. The minimal arrangement is the one that yields minimal expected escape time, and is found using exhaustive search.
The sorted arrangement refers to sorting the weights (transition probabilities) in ascending order.
The random arrangement refers to a random (uniform) arrangement of the given weights. For small values of $\varAlen$ this can be calculated exactly by averaging over all possible arrangements. When this becomes computationally infeasible, we use Monte-Carlo methods to estimate this quantity.

Our first experiment compares the maximum, random, and minimum arrangements. To do so, we consider a random walk in a random environment setting on a line with $\varAlen=8$ nodes. We initialize the environment weights using a uniform distribution on $\brk[s]{0.5-x, 0.5+x}$ and perform a Monte-Carlo simulation (only on the initialization) to evaluate the expected escape time of each arrangement. The results are depicted in \cref{fig: min-max comparison}. Our choice of distribution keeps the expected value of the weights fixed while varying their variance. Unsurprisingly, the arrangement of the weights becomes more significant for higher variance weight initialization. Notice that the graph displays the logarithm of the escape time, and thus the increasing gaps between the arrangements imply a highly super-linear dependence on the variance.

Our second experiment examines the behavior of the escape time as a function of $\varAlen$ for the maximal, sorted and random arrangements (see \cref{fig: d-comparison}). We observe two types of behaviors depending on the properties of the random environment. The first behavior occurs when all weights are smaller than $0.5$, and yields a walk that is, in a sense, ``strongly" transient, making the escape time grow slowly (linearly) in $\varAlen$ regardless of the arrangement. While there is a significant gain in using the maximal (pendulum) and sorted arrangements, which perform similarly here, the overall behavior of the escape time does not change compared to a random arrangement. The second case reveals an interesting phase transition. It considers a case where the random arrangement is transient but some proportion of the weights are greater than $0.5$. In this case the random arrangement behaves as in the first environment (up to small factors). However, starting at some $\varAlen_0$, the maximal and sorted arrangements grow exponentially, with a significant gap between them. We have tried various environment parameters and this behavior seems to persist with the only change being the critical value of $\varAlen_0$ where the change in behavior occurs. We leave the formal investigation of this phenomenon to future work.

\subsection{Alternate Statistics}
In this work we focused on the maximization of the expected escape time. While maximizing the expected value is a highly accepted notion, one could also consider other criteria that, for example, consider some notion of risk.
For instance, one might wish to find a permutation for which
$
f(\mathbb{E} \escapeTimeOf{\permSig\varP}{0}, \text{var}(\escapeTimeOf{\varP \permSig}{0}))
$
is maximized. Some classical examples include the Sharpe Ratio ${f(x, y) = \frac{x}{\sqrt{y}}}$, and Mean-Variance criterion ${f(x,y) = x - \lambda y}$.

An alternative notion that is of separate interest is minimizing the expected escape time.  This problem was studied in a simplified setting where weights are constrained to one of two values, showing that the asymptotic optimal order requires equal spacing between the larger weights \cite{procaccia2012need,lampert2013maximizing}.
In \cref{fig: minimal arrangement} we depict three instantiations of general weight assignments for a line of $\varAlen = 6$ nodes. Contrary to the maximal expected escape time, the minimal optimal permutation is \textit{value dependent}, suggesting that understanding the structure of the minimal permutation is more involved. Extensive simulations lead us to the conjecture that ``large" values are indeed spaced more or less evenly, but it remains unclear how to characterize this notion formally. We leave the topic of alternate statistics as an open question for future work.

\begin{figure}[t!]
\includegraphics[width=0.9\textwidth]{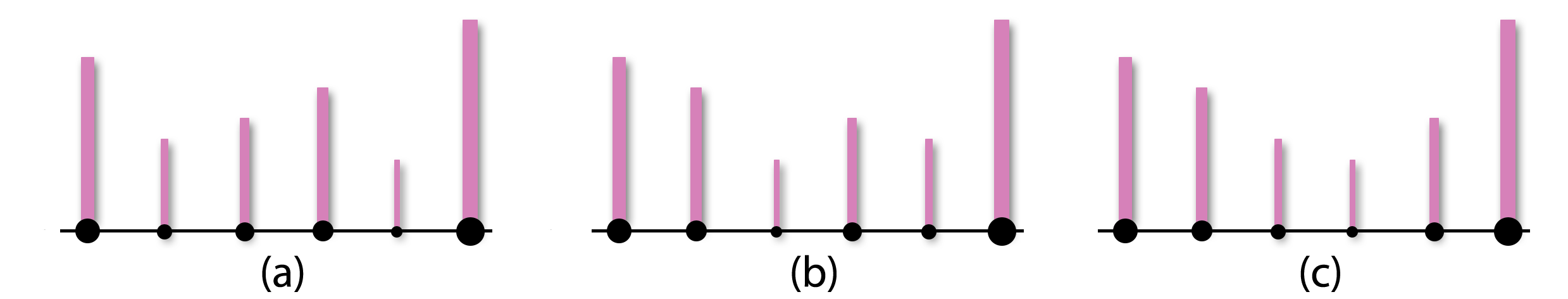}
\centering
\caption{\label{fig: minimal arrangement}
Minimal arrangements of three vectors: (a) $[0.1, 0.2, 0.3, 0.4, 0.5, 0.6]$, (b) $[0.3, 0.4, 0.5, 0.6, 0.7, 0.8]$, and (c) $[0.4, 0.5, 0.6, 0.7, 0.8, 0.9]$. The arrangements are all unique to their values, suggesting that the minimal arrangement depends on the values of $\varP$. Note that the arrangement in (c) is not the inverted pendulum arrangement, as the two largest values in the edges are flipped. }
\end{figure}

\subsection{Conclusion}

In this work we conducted \textit{exact} analysis of a newly discovered phenomenon of heterogeneous random walks. We showed that the maximum escape time is established when the transition probabilities relating to the slowdown drift of the process are ordered in a unique arrangement, known as the pendulum arrangement (see \cref{fig: pendulum arrangement}). Our result follows careful inspection of a sum over products combinatorial optimization problem, which may be of broader interest in fields out of the scope of this paper. 

Finally, our work lays the foundations for Markov chain Design, through careful design of the topology and weights of Markov chains. This may enable the construction of networks that are insusceptible to cyber-attacks, resilient to the spread of infectious diseases, and control the flow of perilous processes (e.g., harmful ideas) on social networks and the web.

\bibliography{bibliography}

\newpage
\appendix

\section*{\centering Appendix: Missing Proofs}

\section{Proof of \cref{proposition:escapeTimeFormula}}
\label{sec:proofOfEscapeFormula}

The proof follows standard induction analysis, (see e.g., Proposition~2 of \cite{barrera2009abrupt}), and is provided here for completeness. For $1 \leq k \leq \varAlen$, due to the Markov property,
\begin{align*}
    \expectedEscape{\varP}{k} 
    = 
    \varPof{k}
    \expectedEscape{\varP}{k-1} 
    +
    (1-\varPof{k})
    \expectedEscape{\varP}{k+1}
    +
    1.
\end{align*}
Rearranging the above yields
\begin{align*}
    \expectedEscape{\varP}{k} - \expectedEscape{\varP}{k+1}
    =
    \frac{\varPof{k}}{1 - \varPof{k}} \brk*{\expectedEscape{\varP}{k-1} - \expectedEscape{\varP}{k}} + \frac{1}{1 - \varPof{k}}.
\end{align*}
Denoting $D_k = \expectedEscape{\varP}{k} - \expectedEscape{\varP}{k+1}$ we get
\begin{align*}
    D_k = \frac{\varPof{k}}{1 - \varPof{k}} D_{k-1} + \frac{1}{1 - \varPof{k}}.
\end{align*}
Solving this equation by iteration yields
\begin{align*}
    D_k
    &=
    \frac{1}{1 - \varPof{k}}
    +
    D_0 \prod_{i=1}^k \frac{\varPof{i}}{1 - \varPof{i}} 
    +
    \sum_{m=1}^{k-1} \frac{1}{1 - \varPof{m}} \prod_{i=m+1}^k \frac{\varPof{i}}{1 - \varPof{i}}.
\end{align*}
Furthermore we have that
\begin{align*}
    &\escapeTimeOf{\varP}{\varAlen+1} = 0\\
    &\escapeTimeOf{\varP}{0} = \escapeTimeOf{\varP}{1} + 1 \Rightarrow D_0 = 1.
\end{align*}
Then, combining the above we get that
\begin{align*}
    \expectedEscape{\varP}{0}
    &=
    1 + \expectedEscape{\varP}{1} \\
    &=
    1
    +
    \sum_{k=1}^{\varAlen}
    \brk*{
    \expectedEscape{\varP}{k} - \expectedEscape{\varP}{k+1}
    } \\
    &=
    1
    +
    \sum_{k=1}^{\varAlen}
    \brk*{
     \frac{1}{1 - \varPof{k}}
    +
    \prod_{i=1}^k \frac{\varPof{i}}{1 - \varPof{i}} 
    +
    \sum_{m=1}^{k-1} \frac{1}{1 - \varPof{m}} \prod_{i=m+1}^k \frac{\varPof{i}}{1 - \varPof{i}}
    }.
\end{align*}
Finally Lemma~\ref{lemma: technical expression derivation} below shows how the final expression can be technically derived from the above, using simple algebraic manipulations.

\begin{lemma}
    \label{lemma: technical expression derivation}
    It holds that
    \begin{align*}
        \sum_{k=1}^{\varAlen}
        \brk*{
         \frac{1}{1 - \varPof{k}}
        +
        \prod_{i=1}^k \frac{\varPof{i}}{1 - \varPof{i}} 
        +
        \sum_{m=1}^{k-1} \frac{1}{1 - \varPof{m}} \prod_{i=m+1}^k \frac{\varPof{i}}{1 - \varPof{i}}
        }
        =
        d
        +
        2
        \sum_{m=1}^{\varAlen}
        \sum_{i=1}^{\varAlen - m + 1} 
        \prod_{j=i}^{i + m - 1} \frac{\varPof{j}}{1-\varPof{j}}
    \end{align*}
\end{lemma}
\begin{proof}
For $z \in (0,1)^{\varAlen}, 1 \leq x \leq y \leq \varAlen$ denote
\begin{align*}
    G_{x,y}(z) = \prod_{i=x}^y z_i, \quad M_{xy}(z) = G_{1,x-1}(1-z)G_{x,y}(z)G_{y+1,\varAlen}(1-z).
\end{align*}
Recalling that $\varP\in (0,1)^{\varAlen}$ denotes the vector of probabilities $(\varPof{1}, \hdots \varPof{\varAlen})$, we have that
\begin{align*}
    \sum_{k=1}^{\varAlen}
        &\brk*{
         \frac{1}{1 - \varPof{k}}
        +
        \prod_{i=1}^k \frac{\varPof{i}}{1 - \varPof{i}} 
        +
        \sum_{m=1}^{k-1} \frac{1}{1 - \varPof{m}} \prod_{i=m+1}^k \frac{\varPof{i}}{1 - \varPof{i}}
    }
    \\
    &=
    \sum_{k=1}^{\varAlen}
    \frac{1}{1-\varPof{k}}
    +
    \sum_{k=1}^{\varAlen}
    \frac{G_{1,k}(\varP)}{G_{1,k}(1-\varP)}
    +
    \sum_{k=1}^{\varAlen}
    \sum_{m=1}^{k-1}
    \frac
    {
        G_{m+1,k}(\varP)
    }
    {
        G_{m,k}(1-\varP)
    } \\
    &=
    \sum_{k=1}^{\varAlen}
    \frac{1}{1-\varPof{k}}
    +
    \frac{1}{G_{1,\varAlen}(1-\varP)}
    \sum_{k=1}^{\varAlen} 
    \brk*{M_{1,k}(\varP) + \sum_{m=1}^{k-1} \brk*{M_{m+1,k}(\varP) + M_{m,k}(\varP)}},
\end{align*}
where in the last two steps we use the definition of $G, M$ and the fact that
\begin{align*}
    &M_{m+1,k}(\varP) + M_{m,k}(\varP) \\
    &=
    G_{1,m}(1-\varP)G_{m+1,k}(\varP)G_{k+1,\varAlen}(1-\varP)
    +
    G_{1,m-1}(1-\varP)G_{m,k}(\varP)G_{k+1,\varAlen}(1-\varP) \\
    &=
    (1 - \varPof{m} + \varPof{m})G_{1,m-1}(1-\varP)G_{m+1,k}(\varP)G_{k+1,\varAlen}(1-\varP) \\
    &=
    G_{1,m-1}(1-\varP)G_{m+1,k}(\varP)G_{k+1,\varAlen}(1-\varP).
\end{align*}
Next, denote
\begin{equation*}
    W_{x,y}(z) 
    = 
    \prod_{i=x}^y\frac{\varPof{i}}{1-\varPof{i}},
\end{equation*}
and notice that
\begin{equation*}
    W_{x,y}(\varP) = \frac{M_{x,y}(\varP)}{G_{1,d}(1-\varP)}
\end{equation*}
Then, we have that
\begin{align*}
    \expectedEscape{\varP}{0}
    &=
    \sum_{k=1}^{\varAlen}
    \frac{1}{1-\varPof{k}}
    +
    \sum_{k=1}^{\varAlen} 
    \brk*{W_{1,k}(\varP) + \sum_{m=1}^{k-1} \brk*{W_{m+1,k}(\varP) + W_{m,k}(\varP)}} \\
    &=
    d
    +
    2
    \sum_{m=1}^{\varAlen} 
    \sum_{i=1}^{\varAlen - m + 1} 
    W_{i, i+m-1}(\varP),
\end{align*}
where the last step is proven by induction on $\varAlen$. Substituting for $W$ completes the proof.

\paragraph{Induction} 
We show that 
\begin{align}
\label{eq: induction}
    \sum_{k=1}^{\varAlen}
    \frac{1}{1-\varPof{k}}
    +
    \sum_{k=1}^{\varAlen} 
    \brk*{W_{1,k}(\varP) + \sum_{m=1}^{k-1} \brk*{W_{m+1,k}(\varP) + W_{m,k}(\varP)}} 
    = 
    d
    +
    2
    \sum_{m=1}^{\varAlen} 
    \sum_{i=1}^{\varAlen - m + 1} 
    W_{i, i+m-1}(\varP),
\end{align}
by induction on $\varAlen$. \\
\textbf{Base case: $\varAlen = 1$.} We have that
\begin{align*}
    \frac{1}{1-\varPof{1}} + W_{1,1}(\varP)
    =
    \frac{1-\varPof{1}+\varPof{1}}{1-\varPof{1}} + \frac{\varPof{1}}{1-\varPof{1}}
    =
    1 + 2 \frac{\varPof{1}}{1-\varPof{1}}
    =
    1 + 2W_{1,1}(\varP).
\end{align*}
\textbf{Induction step.}
Assume \cref{eq: induction} holds for some $\varAlen = n$. We will show it holds for $n+1$ as well. Indeed,
\begin{align*}
    &\sum_{k=1}^{n+1}
    \frac{1}{1-\varPof{k}}
    +
    \sum_{k=1}^{n+1} 
    \brk*{W_{1,k}(\varP) + \sum_{m=1}^{k-1} \brk*{W_{m+1,k}(\varP) + W_{m,k}(\varP)}} 
    \\
    &=
    \frac{1}{1-\varPof{n+1}}
    +
    W_{1,n+1}(\varP) + \sum_{m=1}^{n} \brk*{W_{m+1,n+1}(\varP) + W_{m,n+1}(\varP)} \\
    &~~+
    \sum_{k=1}^{n}
    \frac{1}{1-\varPof{k}} 
    +
    \sum_{k=1}^{n} 
    \brk*{W_{1,k}(\varP) + \sum_{m=1}^{k-1} \brk*{W_{m+1,k}(\varP) + W_{m,k}(\varP)}} \\
    &\underset{(a)}{=}
     \frac{1}{1-\varPof{n+1}}
    +
    W_{1,n+1}(\varP) + \sum_{m=1}^{n} \brk*{W_{m+1,n+1}(\varP) + W_{m,n+1}(\varP)}
    +
    n + 2\sum_{m=1}^{n} 
    \sum_{i=1}^{n - m + 1} 
    W_{i, i+m-1}(\varP) \\
    &\underset{(b)}{=}
    1 + W_{n+1,n+1}
    +
    W_{1,n+1}(\varP) + \sum_{m=1}^{n} \brk*{W_{m+1,n+1}(\varP) + W_{m,n+1}(\varP)}
    +
    n + 2\sum_{m=1}^{n} 
    \sum_{i=1}^{n - m + 1} 
    W_{i, i+m-1}(\varP) \\
    &\underset{(c)}{=}
    n+1
    +
    2\brk*{
     W_{1, n+1}(\varP)
    +
    \sum_{m=1}^{n}
     W_{n-m+2, n+1}(\varP)
     +
    \sum_{m=1}^{n}
    \sum_{i=1}^{n - m + 1} 
    W_{i, i+m-1}(\varP) }\\
    &=
    n + 1
    +
    2\sum_{m=1}^{n + 1}
    \brk*{
     W_{n-m+2, n+1}(\varP)
     +
    \sum_{i=1}^{n - m + 1} 
    W_{i, i+m-1}(\varP)} \\
    &=
    n + 1
    +
    2\sum_{m=1}^{n + 1} 
    \sum_{i=1}^{n - m + 2} 
    W_{i, i+m-1}(\varP).
\end{align*}
    In $(a)$ we used the induction step, in $(b)$ we used the fact that $\frac{1}{1-\varPof{n+1}} = \frac{1-\varPof{n+1} + \varPof{n+1}}{1-\varPof{n+1}} = 1 + W_{n+1,n+1}(\varPof{n+1})$, and in $(c)$ reorganization of the summands.

\end{proof}

\section{Proof of \cref{lemma:sort to pend}}
\label{sec:proofOfSortToPend}
\begin{proof}
    Recall that
    \begin{equation*}
        \permSortToPendOf{j} 
        =
        \begin{cases}
            2j-1 &, j \leq \frac{d+1}{2} \\
            2(d+1-j) &, \text{otherwise}.
        \end{cases}
    \end{equation*}
    It is easy to verify that the inverse of this permutation, i.e., $\permSortToPendInv$, has the following form
    \begin{equation*}
        \permSortToPendInvOf{j}
        =
        \begin{cases}
            \frac{j + 1}{2} &, j \text{ is odd} \\
            d + 1 - \frac{j}{2} &, j \text{ is even}.
        \end{cases}
    \end{equation*}
    Assume that $\sortOf{\varA} = \permSortToPendInv \pendOf{\varA}$ and so by the uniqueness of the sorted order and the permutation $\permSortToPendInv$ we conclude the uniqueness of $\pendOf{\varA}$. Since both sides are now uniquely defined, we can apply $\permSortToPend$ to both sides to obtain the other part of the lemma.
    
    We show that $\sortOf{\varA} = \permSortToPendInv \pendOf{\varA}$ thus concluding the proof.
    Let $y = \pendOf{\varA}$ and $z = \permSortToPendInv y$. Let $1 \le i \le \varAlen - 1$ be odd, then $\permSortToPendInvOf{i} = \brk*{i+1} / 2$ and $\permSortToPendInvOf{i+1} = \varAlen + 1 - \brk*{\brk*{i + 1} / 2}$, and so we have that
    \begin{align*}
        z_i
        =
        y_{\permSortToPendInvOf{i}}
        =
        y_{\brk*{i + 1} / 2}
        \le
        y_{\varAlen + 1 - \brk*{\brk*{i + 1} / 2}}
        =
        y_{\permSortToPendInvOf{i + 1}}
        =
        z_{i + 1},
    \end{align*}
    where the inequality used the first part of \cref{def:pendulum arrangement} (pendulum arrangement). Now, for let $2 \le i \le \varAlen - 1$ be even, then $\permSortToPendInvOf{i} = \varAlen + 1 - \brk*{i / 2}$ and $\permSortToPendInvOf{i + 1} = \brk*{i + 2} / 2$, and so we have that
    \begin{align*}
        z_i
        =
        y_{\permSortToPendInvOf{i}}
        =
        y_{\varAlen + 1 - \brk*{i / 2}}
        \le
        y_{\brk*{i + 2} / 2}
        =
        y_{\permSortToPendInvOf{i + 1}}
        =
        z_{i + 1},
    \end{align*}
    where the inequality used the second part of \cref{def:pendulum arrangement} (pendulum arrangement). Overall we conclude that $z_i \le z_{i + 1}$ for all $1 \le i \le \varAlen - 1$, i.e., $z = \sortOf{\varA}$, as desired.
\end{proof}

\section{Proof of \cref{lemma: impPerm}}
\label{sec:proofOfImpPermLemma}
The proof of \cref{lemma: impPerm} is an immediate corollary of the three following results.
To ease notation, we make the following definition. For $i \in \mathbb{Z}, x \in \RR[\varAlen]$, and $1 \le m \le \varAlen$ let
\begin{align} \label{eq:winDef}
    \winOf[m]{i}{\varA}
    =
    \begin{cases}
        \prod_{j=i}^{i+m-1} \varAof{j} &, 1 \le i \le \varAlen + 1 - m \\
        0 &, \text{ otherwise}
        ,
    \end{cases}
\end{align}
where the otherwise case serves to avoid some edge cases in what follows.
When $\varA$ is clear from context, we will only write $\win[m]{i}$.
The first result, whose proof may be found in \cref{sec:valueDecomp}, decomposes the value.
\begin{lemma}[Value decomposition] \label{lemma:valueDecomp}
    We have that
    \begin{align*}
        \valof
        =
        \sum_{i=1}^{\ceil{\frac{l-m}{2}}} \brk[s]*{\winOf[m]{i}{\varA} + \winOf[m]{(l + 2 - m) - i}{\varA}}
        +
        \sum_{i = l + 2 - m}^{\varAlen + 1 - m} \brk[s]*{\winOf[m]{i}{\varA}}
        +
        \winOf[m]{\frac{1}{2}(l + 2 - m)}{\varA}\indEvent{l-m \in 2 \mathbb{N}}
    \end{align*}
\end{lemma}

The second result, whose proof may be found in \cref{sec:improvingWindowPairs}, shows that the terms in the first sum of the decomposition, as well as the last term, increase as a result of applying $\permImp$.
\begin{lemma}[Improving Window Pairs] \label{lemma:improvingWindows}
    For all $m > 0$ and $1 \le i \le \frac{1}{2}\brk*{l + 2 - m}$, we have that
    \begin{equation*}
        \winOf[m]{i}{\varA} + \winOf[m]{(l + 2 - m) - i}{\varA}
        \le 
        \winOf[m]{i}{\permImp\varA} + \winOf[m]{(l + 2 - m) - i}{\permImp\varA}
        .
    \end{equation*}
    Moreover, if $l=\varAlen$ and $\permImp[\varAlen]\varA \notin \brk[c]*{x, \permInv \varA}$ then there exist $i, m$ such that the inequality is strict.
\end{lemma}

The third and final result, whose proof may be found in \cref{sec:proofOfImprovingSingleWindows}, shows that the terms in the second sum of the decomposition increase as a result of applying $\permImp$.
\begin{lemma}[Improving Single Windows] \label{lemma:ImprovingSingleWindows}
    For all $m > 0$ and $\brk*{l + 2 - m} \le i \le \varAlen + 1 - m$, we have that
    \begin{align*}
        \winOf[m]{i}{\varA}
        \le
        \winOf[m]{i}{\permImp\varA}
        .
    \end{align*}
    Moreover, if $l \le \varAlen - 1$ and $\permImp\varA \neq \varA$ then there exist $i, m$ such that the inequality is strict.
\end{lemma}

\begin{proof}[of \cref{lemma: impPerm}]
    Combining \cref{lemma:valueDecomp,lemma:improvingWindows,lemma:ImprovingSingleWindows} we get that
    \begin{align*}
        &\valof
        \\
        &=
        \sum_{i=1}^{\ceil{\frac{l-m}{2}}} \brk[s]*{\winOf[m]{i}{\varA} + \winOf[m]{(l + 2 - m) - i}{\varA}}
        +
        \sum_{i = l + 2 - m}^{\varAlen + 1 - m} \brk[s]*{\winOf[m]{i}{\varA}}
        +
        \winOf[m]{\frac{1}{2}(l + 2 - m)}{\varA}\indEvent{l-m \in 2 \mathbb{N}}
        \\
        &\le
        \sum_{i=1}^{\ceil{\frac{l-m}{2}}} \brk[s]*{\winOf[m]{i}{\permImp\varA} + \winOf[m]{(l + 2 - m) - i}{\permImp\varA}}
        +
        \sum_{i = l + 2 - m}^{\varAlen + 1 - m} \brk[s]*{\winOf[m]{i}{\permImp\varA}}
        +
        \winOf[m]{\frac{1}{2}(l + 2 - m)}{\permImp\varA}\indEvent{l-m \in 2 \mathbb{N}}
        \\
        &=
        \valof[\permImp\varA]
        .
    \end{align*}
    The strict inequality condition follows by combining those of \cref{lemma:improvingWindows,lemma:ImprovingSingleWindows}.
\end{proof}

\subsection{Proof of \cref{lemma:valueDecomp}}
\label{sec:valueDecomp}

\begin{proof}
    We have that
    \begin{align*}
        \valof[\varA] 
        &= 
        \sum_{i=1}^{\varAlen - m + 1} \prod_{j=i}^{i + m - 1} \varAof{j} \\
        &=
        \sum_{i=1}^{\varAlen - m + 1} \winOf[m]{i}{\varA} \\
        &=
        \sum_{i=1}^{\ceil{\frac{l-m}{2}}} \winOf[m]{i}{\varA} 
        +
        \sum_{i=\ceil{\frac{l-m}{2}} + 1}^{l + 1 - m} \winOf[m]{i}{\varA}
        +
        \sum_{i=l + 2 - m}^{\varAlen - m + 1} \winOf[m]{i}{\varA} \\
        &=
        \sum_{i=1}^{\ceil{\frac{l-m}{2}}} \winOf[m]{i}{\varA} 
        +
        \sum_{i=1}^{\floor{\frac{l-m}{2}} + 1} \winOf[m]{(l + 2 - m) - i}{\varA}
        +
        \sum_{i=l + 2 - m}^{\varAlen - m + 1} \winOf[m]{i}{\varA} \\
        &=
        \sum_{i=1}^{\ceil{\frac{l-m}{2}}} 
        \winOf[m]{i}{\varA} + \winOf[m]{(l + 2 - m) - i}{\varA}
        +
        \sum_{i=l + 2 - m}^{\varAlen - m + 1} \winOf[m]{i}{\varA}
        +
        \winOf[m]{\frac{1}{2}(l + 2 - m)}{\varA}\indEvent{l-m \in 2 \mathbb{N}},
    \end{align*}
    where the last two transitions used the change of variables $i = (l + 2 - m) - j$ and the fact that
    \begin{equation*}
        (l + 1 - m) - \ceil{\frac{l - m}{2}} = \floor{\frac{l - m}{2} + 1} = \ceil{\frac{l-m}{2}} + \indEvent{l-m \in 2 \mathbb{N}}.
    \end{equation*}
\end{proof}

\subsection{Proof of \cref{lemma:improvingWindows}}
\label{sec:improvingWindowPairs}

To prove this lemma, we need a few intermediate results. The first is a simple and well known claim, whose geometric interpretation is that for equal area rectangles, the one with the longest side has a larger circumference. See proof in \cref{sec:proofOfaux1}.
\begin{lemma}
\label{lemma:aux1}
	let $x_1, x_2, y_1, y_2 \ge 0$ such that $x_1 x_2 = y_1 y_2$ then if $\max \brk[c]*{y_1, y_2} < \max \brk[c]*{x_1, x_2}$ then
	\begin{equation*}
	    y_1 + y_2 < x_1 + x_2.
	\end{equation*}
\end{lemma}

The following lemma will imply the condition $x_1 x_2 = y_1 y_2$ of the previous lemma. See proof in \cref{sec:proofOfInvariantWindowPairs}.
\begin{lemma}[Permutation invariant window pairs]
\label{lemma:invariantWindowPairs}
	For all $m > 0$ and $1 \le i \le \frac{1}{2}\brk*{l + 2 - m}$, we have that
    \begin{equation*}
        \winOf[m]{i}{\varA} \winOf[m]{(l + 2 - m) - i}{\varA}
        = 
        \winOf[m]{i}{\permImp\varA} \winOf[m]{(l + 2 - m) - i}{\permImp\varA}.
    \end{equation*}
\end{lemma}

Finally, the following lemma will imply the condition $\max \brk[c]*{y_1, y_2} < \max \brk[c]*{x_1, x_2}$ in \cref{lemma:aux1}. See proof in \cref{sec:proofOfDisjointWindows}.
\begin{lemma}[Improving disjoint windows]
\label{lemma:disjointWindows}
    For all $m > 0$ and $1 \le i \le \frac{1}{2}\brk*{l + 2 - 2m}$ we have that
    \begin{equation*}
        \max\brk[c]*{
            \winOf[m]{i}{\varA},
            \winOf[m]{(l + 2 - m) - i}{\varA}
            }
        \le 
        \winOf[m]{(l + 2 - m) - i}{\permImp\varA}.
    \end{equation*}
    Moreover, for $l = \varAlen$ if $\permImp[\varAlen]\varA \notin \brk[c]*{x, \permInv \varA}$ then there exist $i, m$ such that the inequality is strict.
\end{lemma}

\begin{proof}[of \cref{lemma:improvingWindows}]
    First, notice that the strict inequality condition follows directly from that of \cref{lemma:disjointWindows}. 
    Now, Denote
    \begin{align*}
        y_1 &= \winOf[m]{i}{\varA} \\
        y_2 &= \winOf[m]{(l + 2 - m) - i}{\varA} \\
        x_1 &= \winOf[m]{i}{\permImp\varA} \\
        x_2 &= \winOf[m]{(l + 2 - m) - i}{\permImp\varA}.
    \end{align*}
    Then, by \cref{lemma:invariantWindowPairs}, $x_1 x_2 = y_1 y_2$. We show that $\max\brk[c]*{y_1, y_2} \le x_2$, thus satisfying the requirements of \cref{lemma:aux1} and concluding the proof.
    If $i \le \frac{1}{2}\brk*{l + 2 - 2m}$ then \cref{lemma:disjointWindows} immediately implies the desired. Otherwise, if $i > \frac{1}{2}\brk*{l + 2 - 2m}$ then
    \begin{equation} \label{eq:windowDecomp1}
    \begin{aligned}
        \max\brk[c]*{y_1, y_2}
        &=
        \max\brk[c]*{
            \winOf[m]{i}{\varA},
            \winOf[m]{(l + 2 - m) - i}{\varA}
            } \\
        &=
        \max\brk[c]*{
            \prod_{j=i}^{i + m - 1} \varAof{j},
            \prod_{j=(l + 2 - m) - i}^{l +1 - i} \varAof{j}
            } \\
        &=
        \max\brk[c]*{
            \prod_{j=i}^{(l + 1 - m) - i} \varAof{j},
            \prod_{j=i + m}^{l + 1 - i} \varAof{j}
        }
        \prod_{j = (l + 2 - m) - i}^{i + m - 1}\varAof{j} \\
        &=
        \max\brk[c]*{
            \winOf[m_1]{i}{\varA},
            \winOf[m_1]{(l-m_1+2) - i}{\varA}
            }
        \winOf[m_0]{\frac{1}{2}(l + 2 - m_0)}{\varA},
    \end{aligned}
    \end{equation}
    where $m_0 = 2m + 2i - l - 2$, and $m_1 = m - m_0 = l + 2 - m - 2i$.
    Next, notice that
    \begin{align*}
        &m_0
        >
        2m + (l + 2 - 2m) - l - 2
        =
        0,
    \end{align*}
    and thus taking \cref{lemma:invariantWindowPairs} with $m = m_0$ and $i = \brk*{l + 2 - m} / 2$ we get that
    \begin{equation} \label{eq:windowInvariant1}
        \winOf[m_0]{\frac{1}{2}(l + 2 - m_0)}{\varA}
        =
        \winOf[m_0]{\frac{1}{2}(l + 2 - m_0)}{\permImp \varA}.
    \end{equation}
    Next, notice that
    \begin{align*}
        &m_1
        >
        l + 2 - m - (l + 2 - m)
        =
        0, \\
        &\frac12 \brk*{l + 2 - 2m_1}
        =
        2i
        -
        \frac12 (l + 2 - 2m)
        >
        i
        ,
    \end{align*}
    and thus taking \cref{lemma:disjointWindows} with $m = m_1$ we get that
    \begin{equation} \label{eq:windowImp1}
        \max\brk[c]*{
            \winOf[m_1]{i}{\varA},
            \winOf[m_1]{(l-m_1+2) - i}{\varA}
            }
        \le
        \winOf[m_1]{(l-m_1+2) - i}{\permImp \varA}.
    \end{equation}
    Plugging \cref{eq:windowInvariant1,eq:windowImp1} into \cref{eq:windowDecomp1} we finally get that
    \begin{align*}
        \max\brk[c]*{y_1, y_2}
        \le
        \winOf[m_0]{\frac{1}{2}(l + 2 - m_0)}{\permImp \varA}
        \winOf[m_1]{(l-m_1+2) - i}{\permImp \varA}
        =
        \winOf[m]{(l + 2 - m) - i}{\permImp \varA}
        =
        x_2,
    \end{align*}
    as desired.
\end{proof}

\subsection{Proof of \cref{lemma:ImprovingSingleWindows}}
\label{sec:proofOfImprovingSingleWindows}

\begin{proof}
    We split the proof into three cases according to the value of $i$.
    First, if $i < 1$ then the claim holds trivially since $\winOf[m]{i}{\varA} = 0$ for all $\varA$.
    Second, if $(l + 1) / 2 \le i \le \varAlen + 1 - m$ we have that for all $j \ge i$ $\varAof{\permImpOf{j}} \ge \varAof{j}$ (by definition of $\permImp$) and thus
    \begin{align} \label{eq:impSingleWindow1}
        \winOf[m]{i}{\varA}
        =
        \prod_{j=i}^{i + m - 1} \varAof{j}
        \le
        \prod_{j=i}^{i + m - 1} \varAof{\permImpOf{j}}
        =
        \winOf[m]{i}{\permImp\varA}.
    \end{align}
    Third, if $\max\brk[c]*{1, l + 2 - m} \le i \le \min\brk[c]*{\brk*{l + 1} / 2, \varAlen + 1 - m}$ then letting $m_2 = l + 2 - 2i > 0$, we notice that ${m - m_2 \ge i > 0}$ and so we get that
    \begin{align*}
        \winOf[m]{i}{\varA}
        &=
        \winOf[m_2]{i}{\varA}
        \winOf[m - m_2]{i + m_2}{\varA} \\
        &=
        \winOf[m_2]{\frac12 \brk*{l + 2 - m_2}}{\varA}
        \winOf[m - m_2]{i + m_2}{\varA} \\
        \tag{by \cref{lemma:invariantWindowPairs} with $i = (l + 2 - m_2) / 2$}
        &=
        \winOf[m_2]{\frac12 \brk*{l + 2 - m_2}}{\permImp\varA}
        \winOf[m - m_2]{i + m_2}{\varA} \\
        \tag{*}
        &\le
        \winOf[m_2]{\frac12 \brk*{l + 2 - m_2}}{\permImp\varA}
        \winOf[m - m_2]{i + m_2}{\permImp\varA} \\
        \tag{reversing initial equalities}
        &=
        \winOf[m]{i}{\permImp\varA},
    \end{align*}
    where $(*)$ follows from \cref{eq:impSingleWindow1} since $(l+1)/2 < i + m_2 \le \varAlen + 1 - (m - m_2)$. We covered all the desired values of $i$ thus proving the weak inequality.
    
    Finally, we show the strict inequality condition. If $\permImp \varA \neq \varA$ then there exists $i > (l+1)/2$ such that $\varAof{\permImpOf{i}} > \varAof{i}$ (as in \cref{lemma:disjointWindows}). Taking $i$, $m = \varAlen + 1 - i$, it is trivial to see that the weak inequality in \cref{eq:impSingleWindow1} becomes strict. Notice that $m > 0$, $i \le \varAlen + 1 - m$, and since $l \le \varAlen - 1$ we have that
    $
    i
    \ge
    l + 2 - m
    $.
    We conclude that $i,m$ satisfy the conditions of the lemma and the desired strict inequality.
    
\end{proof}

\subsection{Proof of \cref{lemma:aux1}}
\label{sec:proofOfaux1}
\begin{proof}
    If any of $x_1, x_2, y_1, y_2$ are equal to zero then the claim follows trivially. For the remainder of the proof we assume that $x_1, x_2, y_1, y_2 > 0$.
    Without loss of generality, let $x_2 = \max\brk[c]*{x_1, x_2}$ and $y_2 = \max\brk[c]*{y_1, y_2}$. By the assumptions of the lemma, this implies that $x_2 > y_2 > y_1 > x_1 > 0$.
    Then there exists $\varepsilon > 0$ such that
    \begin{equation}
    \label{eq:aux11}
        x_2
        =
        y_2 + \varepsilon
        .
    \end{equation}
    We then also have that
    \begin{align*}
        x_1
        &=
        \frac{x_1 x_2}{x_2} \\
        \tag{$x_1 x_2 = y_1 y_2$}
        &=
        y_1\frac{y_2}{x_2} \\
        \tag{by \cref{eq:aux11}}
        &=
        y_1 \brk*{1 - \frac{\varepsilon}{x_2}} \\
        \tag{$x_2 > y_1$}
        &>
        y_1 - \varepsilon,
    \end{align*}
    and adding up both results yields the desired.
\end{proof}

\subsection{Proof of \cref{lemma:invariantWindowPairs}}
\label{sec:proofOfInvariantWindowPairs}

\begin{proof}
    Denote the following two sets of indices
    \begin{align*}
        I_1 &= \brk[c]*{i, \ldots, i + m - 1}\\
        I_2 &= \brk[c]*{(l + 2 - m) - i, \ldots, (l+1) - i}.
    \end{align*}
    We will show that $\permImp$ is also a permutation on $I_1 \cup I_2$ and $I_1 \cap I_2$, i.e., $I_1 \cup I_2 = \permImpOf{I_1 \cup I_2}$ and $I_1 \cap I_2 = \permImpOf{I_1 \cap I_2}$, where $\permSigof{I}$ is the result of applying $\permSig$ to each element of $I$. The proof follows immediately since
    \begin{align*}
        \winOf[m]{i}{\varA} \winOf[m]{(l + 2 - m) - i}{\varA}
        &=
        \prod_{j=i}^{i+m-1} \varAof{j} 
        \prod_{j=(l + 2 - m) - i}^{l+1 - i} \varAof{j} \\
        &=
        \prod_{j \in I_1 \cup I_2} \varAof{j} 
        \prod_{j \in I_1 \cap I_2} \varAof{j} \\
        &= 
        \prod_{j \in \permImpOf{I_1 \cup I_2}} \varAof{j}
        \prod_{j \in \permImpOf{I_1 \cap I_2}} \varAof{j}\\
        \tag{$\permImp$ injective}
        &=
        \prod_{j \in I_1 \cup I_2} \varAof{\permImpOf{j}}
        \prod_{j \in I_1 \cap I_2} \varAof{\permImpOf{j}} \\
        &=
        \prod_{j=i}^{i+m-1} \varAof{\permImpOf{j}} 
        \prod_{j=(l + 2 - m) - i}^{l+1 - i} \varAof{\permImpOf{j}}
        =
        \winOf[m]{i}{\permImp\varA} \winOf[m]{(l + 2 - m) - i}{\permImp\varA}.
    \end{align*}
    Since $\permImp$ is a permutation and thus injective, it suffices to show that
    \begin{align}
        \permImpOf{I_1 \cap I_2} &\subseteq I_1 \cap I_2;
        \label{eq:indexIntersection}
        \\
        \permImpOf{I_1 \cup I_2} &\subseteq I_1 \cup I_2.
        \label{eq:indexUnion}
    \end{align}
    Indeed, if $I_1 \cap I_2 = \emptyset$ then \cref{eq:indexIntersection} is trivial. Otherwise, let $j \in I_1 \cap I_2 = \brk[c]*{(l + 2 - m) - i, \ldots, i + m - 1}$. If $\permImpOf{j} = j$ then clearly $\permImpOf{j} \in {I_1 \cap I_2}$. Otherwise $\permImpOf{j} = l + 1 - j$ and we have that
    \begin{align*}
        l + 1 - j
        &\ge
        l + 1 - (i + m - 1)
        =
        (l + 2 - m) - i, \\
        l + 1 - j
        &\le
        l + 1 - (l + 2 - m) + i
        =
        i + m - 1,
    \end{align*}
    thus showing \cref{eq:indexIntersection}. Now for \cref{eq:indexUnion}, let $j \in I_1 \cup I_2$. If $\permImpOf{j} = j$ then clearly $\permImpOf{j} \in {I_1 \cup I_2}$. Otherwise $\permImpOf{j} = l + 1 - j$ and we have the following. If $j \in I_1$ then
    \begin{align*}
        l + 1 - j
        &\ge
        l + 1 - i, \\
        l + 1 - j
        &\le
        l + 1 - (i + m - 1)
        =
        (l + 2 - m) - i,
    \end{align*}
    meaning $\permImpOf{j} \in I_2$. On the other hand, if $j \in I_2$ then
    \begin{align*}
        l + 1 - j
        &\ge
        l + 1 - (l + 1 - i)
        =
        i, \\
        l + 1 - j
        &\le
        l + 1 - (l + 2 - m) + i
        =
        i + m - 1,
    \end{align*}
    meaning $\permImpOf{j} \in I_1$ thus showing \cref{eq:indexUnion} and completing the proof.
\end{proof}

\subsection{Proof of \cref{lemma:disjointWindows}}
\label{sec:proofOfDisjointWindows}

\begin{proof}
    Recalling the definition of $\winOf[m]{i}{\varA}$ in \cref{eq:winDef}, we have that
    \begin{align*}
        \max\brk[c]*{\winOf[m]{i}{\varA}, 
                 \winOf[m]{(l + 2 - m) - i}{\varA}}
        &=
        \max\brk[c]*{\prod_{j=i}^{i+m-1} \varAof{j}, 
                 \prod_{j=(l + 2 - m) - i}^{l+1 - i} \varAof{j}} \\
        &=
        \max\brk[c]*{\prod_{j=(l + 2 - m) - i}^{l+1 - i} \varAof{l+1-j}, 
                 \prod_{j=(l + 2 - m) - i}^{l+1 - i} \varAof{j}}\\
        &\leq
        \prod_{j=(l + 2 - m) - i}^{l+1 - i}
        \max\brk[c]*{\varAof{l+1-j}, \varAof{j}} \\
        &=
        \prod_{j=(l + 2 - m) - i}^{l+1 - i}
        \varAof{\permImpOf{j}} \\
        &=
        \winOf[m]{(l + 2 - m) - i}{\permImp\varA},
    \end{align*}
    where the second to last equality follows from the definition of $\permImp$. To see this, notice that for $j > \frac{l}{2}$ (which is indeed our case since $j \geq (l + 2 - m)-i$ and $i \leq \frac{1}{2}(l+2-2m)$), if $\varAof{j} < \varAof{l+1-j}$, then 
    $
        \varAof{\permImpOf{j}}
        =
        \varAof{l+1-j}
        =
        \max\brk[c]*{\varAof{l+1-j}, \varAof{j}}
        .
    $
    Otherwise, $\varAof{j} \geq \varAof{l+1-j}$ and then 
    $
        \varAof{\permImpOf{j}}
        =
        \varAof{j}
        =
        \max\brk[c]*{\varAof{l+1-j}, \varAof{j}}
        ,
    $
    giving the desired equality.
    
    Now, the weak inequality above becomes strict if and only if there exist 
    $
    j_1, j_2 \in \brk[s]*{(l+2-m)-i, l+1-i}
    $
    such that
    $\varAof{j_1} < \varAof{l+1-j_1}$
    and
    $\varAof{j_2} > \varAof{l+1-j_2}$.
    We show that the strict inequality condition implies the existence of such $j_1, j_2$ thus concluding the proof. Let $l = \varAlen$ and recall that for 
    $j > (\varAlen+1) / 2$,
    $\permImp[\varAlen]$ exchanges $x_j$ and $x_{\varAlen+1-j}$ if and only if 
    $x_j < x_{\varAlen+1-j}$.
    Since 
    $\permImp[\varAlen]\varA \neq \varA$,
    i.e., $\permImp[\varAlen]$ makes an exchange, there exists 
    $j_1 > (\varAlen+1) / 2$
    such that 
    $x_{j_1} < x_{l+1-j_1}$.
    Since 
    $
    \permImp[\varAlen]\varA \neq \permInv \varA
    $,
    there exists $j$ such that
    $
    \varAof{\permImpOf[\varAlen]{j}} \neq \varAof{\varAlen+1-j}
    $
    and since 
    $
    \permImpOf[\varAlen]{j} \in \brk[c]*{j, \varAlen+1-j}
    $
    we have that 
    $\permImpOf[\varAlen]{j} = j$.
    If
    $j > (\varAlen+1) / 2$
    this implies that 
    $\varAof{j} > \varAof{\varAlen+1-j}$
    and so we take $j_2 = j$.
    If 
    $j \le \varAlen / 2$
    then
    $\varAof{j} < \varAof{\varAlen+1-j}$
    and so we take 
    $j_2 = \varAlen+ 1 - j > (\varAlen+1) / 2$. 
    Assume without loss of generality that $j_1 < j_2$ and take 
    $m = 1 + (j_2 - j_1) \ge 2$ 
    and
    $i = (\varAlen + 2 - m) - j_1 = \varAlen + 1 - j_2$. Then
    \begin{align*}
        j_1, j_2
        \in
        \brk[s]*{(\varAlen+2-m)-i, \varAlen+1-i}
        =
        \brk[s]*{j_1, j_2}
        ,
    \end{align*}
    and since $j_1 > (\varAlen + 1) / 2$ we also have that
    $
    i \le (\varAlen + 2 - 2m) / 2
    .
    $
    We conclude that the chosen $i,m$ satisfy the condition for strict inequality, as desired.
    
\end{proof}

\section{Proof of \cref{lemma:pendulumSort}}
\label{sec:proofOfPendSort}

We first need the following lemma whose proof may be found in \cref{sec:proofOfNiExpression}.
\begin{lemma} \label{lemma:NiExpression}
    Let
    $
        \permTilde[\varAlen] = \permSortToPend^{-1} \permImp[\varAlen] \permSortToPend,\;
        \permTilde[\varAlen - 1] = \permSortToPend^{-1} \brk*{\permInv \permImp[\varAlen - 1] \permInv} \permSortToPend
        ,
    $
    where $\permSortToPend$ is from \cref{lemma:sort to pend} and $\permImp$ is from \cref{def:improving perm}, and for $z \in \RR[\varAlen]$, let $N_i\brk*{z}$ be the number of elements in $\brk[c]*{z_{1}, \ldots, z_{i -1}}$ that are strictly greater than $z_{i}$, i.e.,
	\begin{equation*}
		N_i\brk*{z}
		=
		\abs{
			\setDef{j \given j < i \;\land\; z_{j} > z_{i}}}
		.
	\end{equation*}
	We have that
	\begin{align*}
	    N_i(\permTilde[\varAlen] z)
	    \le
	    \begin{cases}
	        N_{\max\brk[c]*{1, i - 1}}(z)&, i \text{ even},\\
	        \max\brk[c]*{N_i(z), N_{i+1}(z) - 1}&, i \text{ odd and } i < \varAlen,\\
	        N_{\varAlen}(z)&, i = \varAlen \text{ odd}
	        ,
	    \end{cases}
	\end{align*}
	\begin{align*}
	    N_i(\permTilde[\varAlen - 1] z)
	    \le
	    \begin{cases}
	        N_{\max\brk[c]*{1, i - 1}}(z)&, i \text{ odd},\\
	        \max\brk[c]*{N_i(z), N_{i+1}(z) - 1}&, i \text{ even and } i < \varAlen,\\
	        N_{\varAlen}(z)&, i = \varAlen \text{ even}
	        .
	    \end{cases}
	\end{align*}
\end{lemma}

\begin{proof}[of \cref{lemma:pendulumSort}]
    Let $\permTilde[\varAlen], \permTilde[\varAlen-1]$ be defined as in \cref{lemma:NiExpression}, and notice that
    \begin{equation*}
        \permSortToPend \brk*{\permTilde[\varAlen - 1] \permTilde[\varAlen]}^k \permSortToPend^{-1}
        =
        \permSortToPend \brk*{ \permSortToPend^{-1} \permInv \permImp[\varAlen - 1] \permInv \permImp[\varAlen] \permSortToPend}^k \permSortToPend^{-1}
        =
        \brk*{\permInv \permImp[\varAlen - 1] \permInv \permImp[\varAlen]}^k.
    \end{equation*}
    Recall that by \cref{lemma:sort to pend} we have that $\permSortToPend \sortOf{\varA} = \pendOf{\varA}$. We show that $\brk*{\permTilde[\varAlen - 1] \permTilde[\varAlen] }^k z = \sortOf{z}$ for all $z \in \RR[\varAlen], k \ge \varAlen / 2$, and then choosing $z = \permSortToPend^{-1} \varA$ concludes the proof.

	To prove the desired we need the following definition. For $z \in \RR[\varAlen]$, let $N_i\brk*{z}$ be the number of elements in $\brk[c]*{z_{1}, \ldots, z_{i -1}}$ that are strictly greater than $z_{i}$. Formally
	\begin{equation*}
		N_i\brk*{z}
		=
		\abs{
			\setDef{j \given j < i \;\land\; z_{j} > z_{i}}}
		.
	\end{equation*}
	Notice that
	\begin{equation} \label{eq:sortedCondition}
		z = \sortOf{z}
		\iff
		N_i\brk*{z} = 0, \;\forall 1 \le i \le \varAlen,
	\end{equation}
	and also that $N_i\brk*{z} \le i - 1$ for all $z \in \RR[\varAlen]$. 
	Now, let $i \in \brk[c]*{1, \ldots, \varAlen}$ be odd, then using \cref{lemma:NiExpression} we have that
	\begin{align*}
	    N_i\brk*{\permTilde[\varAlen - 1]\permTilde[\varAlen] z}
	    \le
	    N_{\max\brk[c]*{1, i - 1}}\brk*{\permTilde[\varAlen] z}
	    \le
	    N_{\max\brk[c]*{1, i - 2}}\brk*{z},
	\end{align*}
	where the second transition used the fact that $i - 1$ is even. Applying this recursively, we get that for $i$ odd and $k \ge 0$
	\begin{equation} \label{eq:NiOdd}
	    N_i\brk*{\brk*{\permTilde[\varAlen-1]\permTilde[\varAlen]}^k z}
	    \le
	    N_{\max\brk[c]*{1, i - 2k}}\brk*{z}
	    \le
	    \max\brk[c]*{0, i - 2k - 1}.
	\end{equation}
	Now, let $i \in \brk[c]*{1,\ldots, \varAlen}$ be even, and split into three cases.
	In the first case, $i = \varAlen$ and thus $\varAlen$ is even. Then using \cref{lemma:NiExpression} we have that
	\begin{align*}
	    N_{i}\brk*{\permTilde[\varAlen - 1]\permTilde[\varAlen] z}
	    =
	    N_{\varAlen}\brk*{\permTilde[\varAlen - 1]\permTilde[\varAlen] z}
	    \le
	    N_{\varAlen}\brk*{\permTilde[\varAlen] z}
	    \le
	    N_{\max\brk[c]*{1, \varAlen - 1}}\brk*{z},
	\end{align*}
	and since here $\varAlen - 1$ is odd, we use \cref{eq:NiOdd} we get that for $i = \varAlen$ even and $k \ge 0$
	\begin{align} \label{eq:NdEven}
	    N_{i}\brk*{\brk*{\permTilde[\varAlen-1]\permTilde[\varAlen]}^k z}
	    \le
	    N_{\max\brk[c]*{1, \varAlen-1}}\brk*{\brk*{\permTilde[\varAlen-1]\permTilde[\varAlen]}^{k-1} z}
	    \le
	    \max\brk[c]*{0, \varAlen - 2k}.
	\end{align}
	In the second case, $i = \varAlen - 1$ and thus $\varAlen$ is odd. Then using \cref{lemma:NiExpression} we have that
	\begin{align*}
	    N_{i}\brk*{\permTilde[\varAlen - 1]\permTilde[\varAlen] z}
	    =
	    N_{\varAlen - 1}\brk*{\permTilde[\varAlen - 1]\permTilde[\varAlen] z}
	    &\le
	    \max\brk[c]*{
	        N_{\varAlen - 1}\brk*{\permTilde[\varAlen] z},
	        N_{\varAlen}\brk*{\permTilde[\varAlen] z} - 1
	        } \\
        &\le
	    \max\brk[c]*{
	        N_{\max\brk[c]*{1, \varAlen - 2}}\brk*{z},
	        N_{\varAlen}\brk*{z} - 1
	        },
	\end{align*}
	and since $\varAlen, \varAlen - 2$ are odd, we can use \cref{eq:NiOdd} to get that for $i = \varAlen - 1$, $i$ even and $k \ge 0$
	\begin{align}
	    \nonumber
	    N_{i}\brk*{\brk*{\permTilde[\varAlen - 1]\permTilde[\varAlen]}^k z}
	    &\le
	    \max\brk[c]*{
	        N_{\max\brk[c]*{1, \varAlen - 2}}\brk*{\brk*{\permTilde[\varAlen - 1]\permTilde[\varAlen]}^{k-1} z},
	        N_{\varAlen}\brk*{\brk*{\permTilde[\varAlen - 1]\permTilde[\varAlen]}^{k-1} z} - 1
	        } \\
        \label{eq:Nd-1Even}
        &\le
        \max\brk[c]*{
            0,
            \varAlen - 2 - 2k + 1,
            \varAlen - 2k
        }
        =
        \max\brk[c]*{0, \varAlen - 2k}.
	\end{align}
	Finally, in the third case, $i \le \varAlen - 2$ is even. Then using \cref{lemma:NiExpression} we have that
	\begin{align*}
	    N_{i}\brk*{\permTilde[\varAlen - 1]\permTilde[\varAlen] z}
	    &\le
	    \max\brk[c]*{
	        N_{i}\brk*{\permTilde[\varAlen] z},
	        N_{i+1}\brk*{\permTilde[\varAlen] z} - 1
	        } \\
        &\le
        \max\brk[c]*{
	        N_{\max\brk[c]*{1, i - 1}}\brk*{z},
	        N_{i+1}\brk*{z} - 1,
	        N_{i+2}\brk*{z} - 2
	        }
	\end{align*}
	Replacing $z$ with $\brk*{\permTilde[\varAlen - 1]\permTilde[\varAlen]}^{k-1} z$ and applying \cref{eq:NiOdd} we get that for $k \ge 0$
	\begin{align*}
	    N_{i}\brk*{\brk*{\permTilde[\varAlen - 1]\permTilde[\varAlen]}^{k} z}
	    &\le
	    \max\brk[c]*{
	        0,
	        i - 1 - 2k + 1,
	        i + 1 - 2k,
	        N_{i+2}\brk*{\brk*{\permTilde[\varAlen - 1]\permTilde[\varAlen]}^{k-1} z} - 2
	    } \\
	    &\le
	    \max\brk[c]*{
	        0,
	        \varAlen - 2k,
	        N_{i+2}\brk*{\brk*{\permTilde[\varAlen - 1]\permTilde[\varAlen]}^{k-1} z} - 2
	    }
	    .
	\end{align*}
	Now, let $k \ge \varAlen / 2$ and let $k_i = \ceil{\brk*{\varAlen - i - 1} / 2}$. We open the recursion above $k_i$ times to get that for $k \ge \varAlen / 2$
	\begin{align}
	    \nonumber
	    N_{i}\brk*{\brk*{\permTilde[\varAlen - 1]\permTilde[\varAlen]}^{k} z}
	    &\le
	    \max\brk[c]*{
	        0,
	        \varAlen - 2k,
	        N_{i+2k_i}\brk*{\brk*{\permTilde[\varAlen - 1]\permTilde[\varAlen]}^{k-k_i} z} - 2k_i
	    } \\
	    \label{eq:NiEven}
	    &\le
	    \max\brk[c]*{
	        0,
	        \varAlen - 2k,
	        \varAlen - 2(k - k_i) - 2k_i
	    }
	    =
	    \max\brk[c]*{
	        0,
	        \varAlen - 2k,
	    },
	\end{align}
	where the second to last transition follows using \cref{eq:NdEven,eq:Nd-1Even} since $i + 2k_i \in \brk[c]*{\varAlen - 1, \varAlen}$.
	Combining \cref{eq:NiOdd,eq:NdEven,eq:Nd-1Even,eq:NiEven} with $k \ge \varAlen / 2$ we conclude that $N_{i}\brk*{\brk*{\permTilde[\varAlen - 1]\permTilde[\varAlen]}^{k} z} = 0$ for all $1 \le i \le \varAlen$ and thus by \cref{eq:sortedCondition} that $\brk*{\permTilde[\varAlen - 1]\permTilde[\varAlen]}^{k} z = \sortOf{z}$.
\end{proof}

\subsection{Proof of \cref{lemma:NiExpression}} \label{sec:proofOfNiExpression}
We first need the following lemma whose proof may be found in \cref{sec:proofOfPermTildeExpression}.
\begin{lemma} \label{lemma:permTildeExpression}
    Define 
    $
        \permTilde[\varAlen] = \permSortToPend^{-1} \permImp[\varAlen] \permSortToPend,
        \permTilde[\varAlen - 1] = \permSortToPend^{-1} \brk*{\permInv \permImp[\varAlen - 1] \permInv} \permSortToPend
        ,
    $
    where $\permSortToPend$ is from \cref{lemma:sort to pend} and $\permImp$ is from \cref{def:improving perm}. Then we have that
    \begin{align*}
        \permTildeOf[\varAlen]{i} = 
        \begin{cases}
            i - 1 & , i > 1 \land i \text{ even } \land \varAof{i} < \varAof{i - 1} \\
            i + 1 &, i < \varAlen \land i \text{ odd } \land \varAof{i} > \varAof{i + 1} \\
            i &, \text{otherwise}
        \end{cases}
    \end{align*}
    \begin{align*}
        \permTildeOf[\varAlen - 1]{i} = 
        \begin{cases}
            i - 1 & , i > 1 \land i \text{ odd } \land \varAof{i} < \varAof{i - 1} \\
            i + 1 &, i < \varAlen \land i \text{ even } \land \varAof{i} > \varAof{i + 1} \\
            i &, \text{otherwise}
        \end{cases}
    \end{align*}
\end{lemma}

\begin{proof}[of \cref{lemma:NiExpression}]
    We prove the expression for $\permTilde[\varAlen]$. The proof for $\permTilde[\varAlen - 1]$ is identical.
    Throughout the proof we treat $\permTilde[\varAlen]$ as the expression derived for it in \cref{lemma:permTildeExpression}.
    First, notice that for $j < i$ we have that $\permTildeOf[\varAlen]{j} \le i$. Moreover, if $i$ is odd then $\permTildeOf[\varAlen]{i-1} \le i - 1$ and thus $\permTildeOf[\varAlen]{j} < i$. We conclude that
    \begin{align} \label{eq:permTildeSubset}
        \setDef{\permTildeOf[\varAlen]{j} \given j < i}
        \subseteq
        \begin{cases}
            \setDef{j \given j < i}&, i \text{ odd} \\
            \setDef{j \given j \le i}&, \text{otherwise}.
        \end{cases}
    \end{align}
    Using the above, we have that for any $z \in \RR[\varAlen]$
    \begin{align*}
        N_i\brk*{\permTilde[\varAlen] z}
        &=
        \abs{\setDef{j \given j < i \land z_{\permTildeOf[\varAlen]{j}} > z_{\permTildeOf[\varAlen]{i}}}} \\
        &=
        \abs{\setDef{\permTildeOf[\varAlen]{j} \given j < i \land z_{\permTildeOf[\varAlen]{j}} > z_{\permTildeOf[\varAlen]{i}}}} \\
        \tag{by \cref{eq:permTildeSubset}}
        &\le
        \begin{cases}
            \abs{\setDef{j \given j < i \land z_{j} > z_{\permTildeOf[\varAlen]{i}}}}&, i \text{ odd} \\
            \abs{\setDef{j \given j \le i \land z_{j} > z_{\permTildeOf[\varAlen]{i}}}}&, \text{otherwise}
        \end{cases} \\
        \tag{by \cref{lemma:permTildeExpression}}
        &\le
        \begin{cases}
            \abs{\setDef{j \given j \le i \land z_{j} > z_{i-1}}}&, i >1 \land i \text{ even} \land z_i < z_{i-1} \\
            \abs{\setDef{j \given j < i \land z_{j} > z_{i+1}}}&, i < d \land i \text{ odd} \land z_i > z_{i+1} \\
            \abs{\setDef{j \given j \le i \land z_{j} > z_{i}}}&, \text{otherwise}.
        \end{cases}
    \end{align*}
    Notice that if $i > 1$ and $z_i < z_{i-1}$ then
    $$
        \abs{\setDef{j \given j \le i \land z_{j} > z_{i-1}}}
        =
        \abs{\setDef{j \given j < i - 1 \land z_{j} > z_{i-1}}}
        =
        N_{i-1}\brk*{z}
        ,
    $$
    and if $i < d$ and $z_i > z_{i+1}$ then
    $$
        \abs{\setDef{j \given j < i \land z_{j} > z_{i+1}}}
        =
        \abs{\setDef{j \given j < i + 1 \land z_{j} > z_{i+1}}} - 1
        =
        N_{i+1}\brk*{z} - 1
        ,
    $$
    and finally that
    $$
        \abs{\setDef{j \given j \le i \land z_{j} > z_{i}}}
        =
        \abs{\setDef{j \given j < i \land z_{j} > z_{i}}}
        =
        N_i\brk*{z}
        .
    $$
    Plugging these back into the above inequality we get that
    \begin{align*}
        N_i\brk*{\permTilde[\varAlen] z}
        \le
        \begin{cases}
            N_{i-1}\brk*{z}&, i >1 \land i \text{ even} \land z_i < z_{i-1} \\
            N_{i+1}\brk*{z} - 1&, i < d \land i \text{ odd} \land z_i > z_{i+1} \\
            N_i\brk*{z}&, \text{otherwise}.
        \end{cases}
    \end{align*}
    Now, if $z_i \ge z_{i-1}$ then
    \begin{align*}
        N_i\brk*{z}
        =
        \abs{\setDef{j \given j < i \land z_{j} > z_{i}}}
        &=
        \abs{\setDef{j \given j < i - 1 \land z_{j} > z_{i}}} \\
        &\le
        \abs{\setDef{j \given j < i - 1 \land z_{j} > z_{i - 1}}}
        =
        N_{i-1}\brk*{z},
    \end{align*}
    and using this fact, and some manipulations on the cases of the previous inequality, we conclude that
    \begin{align*}
        N_i\brk*{\permTilde[\varAlen] z}
        \le
        \begin{cases}
            N_{i-1}\brk*{z}&, i \text{ even} \\
            \max\brk[c]*{N_i\brk*{z}, N_{i+1}\brk*{z} - 1}&, i < d \land i \text{ odd} \\
            N_{\varAlen}\brk*{z}&, i = \varAlen \text{ odd}.
        \end{cases}
    \end{align*}
    Since for $i$ even we have that $i-1 = \max\brk[c]*{1, i-1}$, the proof is concluded.
\end{proof}

\subsection{Proof of \cref{lemma:permTildeExpression}} \label{sec:proofOfPermTildeExpression}
\begin{proof}
    Recall that $\permImp$ is defined w.r.t. the vector it permutes. \\
    Specifically, we have that ${(\permImp[\varAlen] \permSortToPend \varA)(i) = \varAof{\permSortToPendOf{\permImp[\varAlen]^{\permSortToPend\varA}(i)}}}$, where we used $\permImp[\varAlen]^{\permSortToPend\varA}$ to denote $\permImp[\varAlen]$ w.r.t. the vector it permutes, i.e., w.r.t. $\permSortToPend\varA$. 
    We have that
    \begin{align*}
        \permImp[\varAlen]^{\permSortToPend\varA}
        &=
        \begin{cases}
    	{d + 1 - i}, &\quad \text{or}~ 
    	\begin{aligned}
    		&\varAof{\permSortToPendOf{i}} > \varAof{\permSortToPendOf{d + 1 - i}},~ i \le \varAlen / 2 \\
    		&\varAof{\permSortToPendOf{i}} < \varAof{\permSortToPendOf{d + 1 - i}},~ \varAlen / 2 < i \le \varAlen
    	\end{aligned} \\
    	{i}, &\quad \text{otherwise}
    	\end{cases} \\
    	&=
    	\begin{cases}
    	{d + 1 - i}, &\quad \text{or}~ 
    	\begin{aligned}
    		&\varAof{2i-1} > \varAof{2i},~ i \le \varAlen / 2 \\
    		&\varAof{2(\varAlen + 1 - i)} < \varAof{2(d+1-i) - 1},~ \varAlen / 2 < i \le \varAlen
    	\end{aligned} \\
    	{i}, &\quad \text{otherwise}
    	\end{cases}
    \end{align*}
    To prove the lemma, we will show that $\permSortToPend \permTilde[\varAlen]  =  \permImp[\varAlen]\permSortToPend$, i.e., $\permSortToPend \permTilde[\varAlen]^x  =  \permImp[\varAlen]^{\permSortToPend\varA}\permSortToPend$. 
    
    Indeed,
    \begin{align*}
        \permImp[\varAlen]\permSortToPendOf{i}
        &=
        \begin{cases}
            2\permImpOf[\varAlen]{i}-1 &, \permImpOf[\varAlen]{i} \leq \frac{\varAlen+1}{2} \\
            2(\varAlen+1-\permImpOf[\varAlen]{i}) &, \permImpOf[\varAlen]{i} > \frac{\varAlen+1}{2}
        \end{cases} \\
        &=
        \begin{cases}
            2(\varAlen + 1 - i) - 1 &, \varAof{2(\varAlen + 1 - i)} < \varAof{2(d+1-i) - 1},~ \varAlen / 2 < i \le \varAlen \\
            2i &, \varAof{2i-1} > \varAof{2i},~ i \le \varAlen / 2 \\
            \permSortToPendOf{i} &, \text{otherwise}
        \end{cases} 
    \end{align*}
    and
    \begin{align*}
        \permSortToPend \permTilde[\varAlen] (i)
        &=
        \begin{cases}
                \permSortToPendOf{i} - 1 & , \permSortToPendOf{i} > 1 \land \permSortToPendOf{i} \text{ even } \land \varAof{\permSortToPendOf{i}} < \varAof{\permSortToPendOf{i} - 1} \\
                \permSortToPendOf{i} + 1 &, \permSortToPendOf{i} < \varAlen \land \permSortToPendOf{i} \text{ odd } \land \varAof{\permSortToPendOf{i}} > \varAof{\permSortToPendOf{i} + 1} \\
                \permSortToPendOf{i} &, \text{otherwise}
        \end{cases} \\
        &=
        \begin{cases}
                2(\varAlen +1 - i) - 1 & , 2(\varAlen +1 - i) > 1 \land i > \frac{\varAlen+1}{2} \land \varAof{2(\varAlen +1 - i)} < \varAof{2(\varAlen +1 - i) - 1} \\
                2i - 1 + 1 &, 2i - 1 < \varAlen \land i \leq \frac{\varAlen+1}{2} \land \varAof{2i - 1} > \varAof{2i - 1 + 1} \\
                \permSortToPendOf{i} &, \text{otherwise}
        \end{cases} \\
        &=
        \begin{cases}
            2(\varAlen + 1 - i) - 1 &, \varAof{2(\varAlen + 1 - i)} < \varAof{2(d+1-i) - 1},~ \varAlen / 2 < i \le \varAlen \\
            2i &, \varAof{2i-1} > \varAof{2i},~ i \le \varAlen / 2 \\
            \permSortToPendOf{i} &, \text{otherwise},
        \end{cases}
    \end{align*}
    thus $\permSortToPend \permTilde[\varAlen]  =  \permImp[\varAlen]\permSortToPend$.
\end{proof}

\section{Proof of \cref{prop: extreme points}}
\label{appendix: proof of prop extreme points}

\begin{proof}
    The proof follows by the convexity of $\expectedEscape{\varP}{0}$ on $C \subseteq \left[\frac{1}{2}, 1\right)^{\varAlen}$, as shown in \cref{lemma: convexity} below. Since $\expectedEscape{\varP}{0}$ is convex, there exists a maximizer $\varP^* \in \text{ext}(C)$, and by assumption we also have that $\pendOf{\varP^*} \in \text{ext}(C)$. By \cref{thm:main} we have that 
    $
    \expectedEscape{\pendOf{\varP^*}}{0}
    \ge
    \expectedEscape{{\varP^*}}{0}
    $
    and thus
    $
    \pendOf{\varP^*}
    \in
    \pendOf{(\text{ext}(C))}
    $
    is also a maximizer.
\end{proof}

\begin{lemma}[Escape Time Convexity]
\label{lemma: convexity}
$\expectedEscape{\varP}{0}$ is convex for $\varP \in \left[\frac{1}{2}, 1\right)^{\varAlen}$.
\end{lemma}
\begin{proof}
Define $f(a) = \expectedEscape{\frac12 + a}{0}$, then by \cref{proposition:escapeTimeFormula} we have that
\begin{equation*}
    f(a)
    =
    (\varAlen+1)
    +
    \sum_{m=1}^{\varAlen} 
    \sum_{i=1}^{\varAlen - m + 1}
    \prod_{j=i}^{i+m-1}
    \frac{\frac{1}{2}+a_j}{\frac{1}{2}-a_j}.
\end{equation*}
Since this is a linear variable exchange, it suffices to show that $f$ is convex over $[0, \frac12)^{\varAlen}$.
Denote
\begin{align*}
    g_{i,m}(x) = \prod_{j=i}^{i+m-1}
    \frac{\frac{1}{2}+x_j}{\frac{1}{2}-x_j}.
\end{align*}
Then
\begin{equation*}
    f(a)
    =
    (\varAlen+1)
    +
    \sum_{m=1}^{\varAlen} 
    \sum_{i=1}^{\varAlen - m + 1}
    g_{i,m}(a).
\end{equation*}
It is thus enough to show that $g_{i,m}$ are convex in $[0, \frac12)^{\varAlen}$. We use Theorem~3.2 of \cite{crouzeix1995survey} which states that $g_{i,m}$ is convex if and only if $\frac{\frac{1}{2}+x}{\frac{1}{2}-x}$ is log-convex for $x \in [0, \frac12)$. Indeed, 
\begin{equation*}
    \frac{\partial^2}{\partial x^2} \brk*{\log \frac{\frac{1}{2}+x}{\frac{1}{2}-x}} 
    = \frac{2x}{\brk*{\frac{1}{2}-x}^2\brk*{\frac{1}{2}+x}^2} 
    \geq 
    0,\;
    \forall x \in \left[0, \frac{1}{2}\right).
\end{equation*}
\end{proof}

\end{document}